\documentclass[12pt,letterpaper,english]{amsart}
\usepackage{lmodern}
\usepackage{helvet}

\usepackage[T1]{fontenc}
\usepackage[latin9]{inputenc}
\usepackage{mathrsfs}
\usepackage{amsthm}
\usepackage{amstext}
\usepackage{amssymb}
\usepackage{tikz}
\usepackage{esint}
\usepackage{graphicx}
\usepackage{cancel}
\usepackage{comment}
\usepackage{enumitem}
\DeclareFontFamily{OT1}{pzc}{}
\DeclareFontShape{OT1}{pzc}{m}{it}{<-> s * [1.10] pzcmi7t}{}
\DeclareMathAlphabet{\mathpzc}{OT1}{pzc}{m}{it}
\usepackage[bbgreekl]{mathbbol}
\DeclareSymbolFontAlphabet{\mathbb}{AMSb}
\DeclareSymbolFontAlphabet{\mathbbl}{bbold}

\DeclareFontFamily{U}{wncyr10}{}
\DeclareFontShape{U}{wncyr10}{m}{n}{<->wncyr10}{}
\DeclareSymbolFont{mcy}{U}{wncyr10}{m}{n}
\DeclareMathSymbol{\bE}{\mathord}{mcy}{"42}
\DeclareMathSymbol{\e}{\mathord}{mcy}{"62}

\makeatletter

\pdfpageheight\paperheight
\pdfpagewidth\paperwidth

\usepackage{fullpage}
\usepackage{setspace}
\newcommand{\noun}[1]{\textsc{#1}}

\numberwithin{equation}{section}
\numberwithin{figure}{section}
\usepackage{enumitem}		
\theoremstyle{plain}
\newtheorem{thm}{\protect\theoremname}[section]
\theoremstyle{remark}
\newtheorem{rem}[thm]{\protect\remarkname}
\theoremstyle{definition}
\newtheorem{defn}[thm]{\protect\definitionname}
\theoremstyle{definition}
\newtheorem{properties}[thm]{\protect\propertiesname}
\theoremstyle{definition}
\newtheorem{example}[thm]{\protect\examplename}
\theoremstyle{plain}

\theoremstyle{plain}
\newtheorem{prop}[thm]{Proposition}
\theoremstyle{plain}

\theoremstyle{definition}
\newtheorem*{bc}{\protect\bcname}

\usepackage{amsfonts}
\usepackage{amssymb}
\usepackage[all]{xy}
\usepackage{array}
\usepackage{lmodern}
\usepackage[T1]{fontenc}
\usepackage{bm}

\def\uo{\underline{0}}
\def\ux{\underline{x}}
\def\QQ{\mathbb{Q}}
\def\RR{\mathbb{R}}
\def\CC{\mathbb{C}}
\def\ZZ{\mathbb{Z}}
\def\PP{\mathbb{P}}
\def\MM{\mathbb{M}}
\def\GG{\mathbb{G}}
\def\W{\mathcal{W}}
\def\VV{\mathbb{V}}
\def\F{\mathcal{F}}
\def\V{\mathcal{V}}
\def\X{\mathcal{X}}

\def\M{\mathcal{M}}
\def\H{\mathcal{H}}
\def\HH{\mathbb{H}}
\def\D{\mathcal{D}}
\def\ay{\mathbf{i}}
\def\ve{\varepsilon}
\def\cx{\mathcal{X}}
\def\IH{\mathrm{IH}}
\def\co{\mathcal{O}}
\def\ord{\mathrm{ord}}

\def\gr{\mathrm{Gr}}
\def\fa{\mathfrak{a}}
\def\fb{\mathfrak{b}}
\def\E{\mathcal{E}}
\def\EE{\mathbb{E}}
\def\J{\mathcal{J}}
\def\xan{X_{\CC}^{\text{an}}}
\def\DR{\mathrm{DR}}
\def\fx{\mathfrak{X}}
\def\fz{\mathfrak{Z}}
\def\rbe{\mathfrak{r}}
\def\hx{\hat{x}}
\def\hz{\hat{z}}
\def\qb{\bar{\QQ}}
\def\tx{\tilde{X}}
\def\st{\hat{t}}
\def\bbg{\mathbbl{\Gamma}}
\def\bba{\mathbf{A}}
\def\bbb{\mathbf{B}}
\def\bbc{\mathbf{C}}
\def\bbd{\mathbf{D}}
\def\ta{\widetilde{\bba}}
\def\tb{\widetilde{\bbb}}
\def\tc{\widetilde{\bbc}}
\def\td{\widetilde{\bbd}}
\def\th{\hat{t}}
\def\bxr{\square_{\text{rel}}}
\def\dlog{\mathrm{dlog}}
\def\Alpha{\mathrm{A}}

\def\CH{\mathrm{CH}}
\def\Spec{\mathrm{Spec}}
\def\Gm{\mathbb{G}_m}
\def\sm{\setminus}

\newenvironment{psmatrix}
  {\left(\begin{smallmatrix}}
  {\end{smallmatrix}\right)}

\newcommand\Exp[1]{e^{#1}}
\newcommand{\mystrut}{{\vrule depth 1em width 0em height 2em}}

\theoremstyle{definition}
\newtheorem*{thx}{Acknowledgments}

\makeatother

\usepackage{babel}
\providecommand{\corollaryname}{Corollary}
\providecommand{\definitionname}{Definition}
\providecommand{\remarkname}{Remark}
\providecommand{\theoremname}{Theorem}
\providecommand{\examplename}{Example}
\providecommand{\propertiesname}{Properties}

\providecommand{\definitionname}{Question}
\providecommand{\bcname}{Beilinson Conjecture I}

\begin{document}

\title[Calabi-Yau Regulators]{The arithmetic of Calabi-Yau motives \\ and mobile higher regulators}

\author[Golyshev and Kerr]{Vasily Golyshev and Matt Kerr}

%\subjclass[2000]{14C30, 14D07, 19E15, 32G20, 32S40}
\begin{abstract}
We construct elements in the motivic cohomology of certain rank 4 weight 3 Calabi--Yau motives, and write down explicit expressions for the regulators of these elements in the context of conjectures on $L$-values such as those of Beilinson or Bloch-Kato. We apply a combination of three ideas: (i) that a motive can be made to vary in a family in such a way that a desired motivic cohomology class is realized by relative cohomology; (ii)~that there are ways to construct higher-rank (such as $2\times 2$) regulators from a single family; and (iii) that one can arrange elements in $H^4_{\text{Mot}}(X,\mathbb{Z}(p))$ with different $p$'s by choosing hypergeometric families with different local exponents.

Following background material on Hodge theory, algebraic cycles, differential equations, and hypergeometric variations, we work out two cases in detail where $p=3,4$.  Regarding our Calabi-Yau motives $X_t$ as fibers in a suitable total space $\mathcal{X}_U\to U\subset \mathbb{P}^1$, each Hodge class in $\mathrm{Hom}_{\mathrm{MHS}}(\mathbb{Q}(0),H^4(\mathcal{X}_U,\mathbb{Q}(p)))$ produces a family of extension classes in $\mathrm{Ext}^1_{\mathrm{MHS}}(\mathbb{Q}(0),H^3(X_t,\mathbb{Q}(p)))$ called a normal function. Our main results for these cases, which are essentially independent, are the explicit computation of the normal functions, and the construction of motivic cohomology cycles realizing the Hodge classes, thereby proving Beilinson's Hodge-type conjecture and providing the first numerical checks of Beilinson's conjectures on special values of $L$-functions for such motives.
\end{abstract}
\maketitle

\section{Introduction}\label{Intro}

Algebraic cycles are central to modern arithmetic geometry, one of whose main principles is that one should seek to linearize the category $\mathcal{V}ar$ by imposing structures on cohomologies of algebraic varieties that make them similar to group representations. The bigger this group is, the more information about varieties and their morphisms is retained by this linear simplification. Ideally, the group should be of such sophistication as to force any module over it to be realized as a piece of cohomology of an algebraic variety, and a morphism of two modules to arise from a morphism in  $\mathcal{V}ar$, with the dream of ultimately constructing an overarching ``motivic Galois'' group solving this linearization problem in a universal way. 

Yet it would be useless to try to characterize an algebraic cycle class $Z$, say of codimension $d$, by how the corresponding vector is acted upon in such a representation, because by the very definition of a Weil cohomology theory it should stay essentially invariant. For instance, the $p$--Frobenius should act on the respective cohomology class in $l$--adic cohomology by multiplying it by $p^d$, while the integral of the respective de Rham form over the manifold $Z (\CC)$ gives $(2 \pi \ay)^d$. In order to not merely distinguish between two classes $Z_1$ and $Z_2$, but do so in a quantitative way, one turns to  ``from--to'', or Abel--Jacobi, integration patterns. For instance, there is little of interest in what integration over a point --- say, $\Spec(\QQ(\sqrt{5}))$ --- can give; but positioning it as a $0$--cycle in $\Gm$ with coordinate the golden ratio, and integrating the standard differential form $\frac{dz}{z}$ from the origin along various paths bounding on it, yields $\log (\frac{1+\sqrt{5}}{2}) + 2 \pi i \, \ZZ.$ Even though these integrals are not uniquely defined, with the $H^1$ of $\Gm$ itself measuring the ambiguity, their real part \emph{is}, and its significance with respect to the original $0$--dimensional scheme is that    
$\sum_{n=1}^\infty \frac{\chi_5 (n)}{n} = \frac{2}{\sqrt{5}} \log (\frac{1+\sqrt{5}}{2})$, or, equivalently, $L'(\QQ (\sqrt{5}),0) = - \frac{1}{2} \log (\frac{1+\sqrt{5}}{2})$. 

According to Beilinson's conjectures, this should extend to other motives. Bloch's higher cycles are designed exactly in such a way as to give rise to an integration pattern where an Abel--Jacobi--type integral\footnote{This refers to the \emph{regulator periods} we shall encounter below, which come from evaluating a well-chosen ``lift to cohomology'' of the Abel--Jacobi class against a Betti class instead of a de Rham one.  Strictly speaking, even these often have additional ambiguities by ``irrelevant periods'' (irrelevant in the sense that they are killed in the regulator determinant), but in the families we consider the choice of lift in Theorem \ref{t3a} will always eliminate these.} is in a natural way defined only up to $(2 \pi \ay)^m$ for some $m \in  \ZZ$.  Starting from a complete system of independent cycles, one produces matrices whose entries are the integrals arising in this fashion; $L$--values are then expected to be proportional to their determinants. The farther away from the center of the critical strip the $L$--argument is, the more extrinsic dimensions for Abel--Jacobi integration are needed, and the more involved the counterpart of $\frac{dz}{z}$ becomes:  think  of the factor $\mathrm{Li}_2 (\theta), \; \theta= \exp (2 \pi i/3)$ in the RHS of
$$ \zeta(2) \sum_{n=1}^\infty \frac{\chi_{-3} (n)}{n^2} =:  \zeta (\QQ (\sqrt{-3}),2) = \frac{\sqrt{3}}{9} \pi ^2\, \mathop{\mathrm{Im}}\sum_{n=1}^\infty \frac{\theta^n}{n^2} $$
as arising from integrating  a standard current along the $1$--cycle $Z_3$ in $(\PP^1\setminus \{ 1 \})^3$
given by
the parametrized curve $(1-\frac{\theta}{t},1-t,t^{-3})$: %-(1-\frac{\theta^{-1}}{t},1-t,t^{-3})$  of a standard current 
$$ (2\pi \ay) \; \mathrm{Li}_2 (\theta) = \int_{Z_3}
\log z_1 \frac{dz_2}{z_2} \wedge \frac{dz_3}{z_3} +
(2\pi \ay) \log z_2 \frac{dz_3}{z_3} \cdot \delta_{\{z_1 \in \RR_{\le 0}\}}+
(2\pi \ay)^2 \log z_3 \cdot \delta_{\{z_1 \in \RR_{\le 0}\} \cap \{z_2 \in \RR_{\le 0}\}}
,$$
where $\delta_{\mathcal{C}}$ means the current of integration over the chain $\mathcal{C}$.\footnote{In fact, the first and third terms pull back to zero on $Z_3$.  See \S\ref{Sell} for details of a similar example.}

The subject of this paper is Hodge theory of Calabi--Yau motives, insofar as it is needed for computing regulators, and to get to it we invert the logic just explained. As a vague principle, exhibiting an ``integration scheme'' (i.e., a system of films and integrands on them) that produces an extension period matrix $\Pi'(M)$ whose blocks are the pure period matrix $\Pi (M)$ of the motive $M$ in question and a $1\times 1$ block with the entry $(2\pi \ay)^m$ in it --- as long as this integration is of algebro-geometric nature --- will signal the presence of a higher cycle whose regulator is expected to be an expression in the entries of $\Pi'$.  It is not even important whether or not we are actually able to construct such a cycle, as the only thing we care about is the regulator itself. 

\subsection*{(i) Cycles from deformations of CY motives}
We can make this very specific by introducing three more ideas. One is that making a motive vary in a family can produce an extension as above: with some luck, one can arrange the 1x1 block to come from the total cohomology of the family, which can be much simpler than the motives in question. (Imagine $\mathbb{A}^n$ being fibered out over $\mathbb{A}^1$ by a random polynomial with its complicated level hypersurfaces.) To deform our $L'(\QQ (\sqrt{5}),0)$ example, in \S\ref{SCY0} we consider the family of ``CY $0$-folds'' $X_t=\{x_+(t),x_-(t)\}$ defined, for $t\in \CC\sm\{0,\tfrac{1}{4}\}$, by $x^2-(t^{-1}-2)x+1=0$.  This yields a hypergeometric regulator expression $\log(t)+\sum_{k>0} \binom {2k}{k}\tfrac{t^k}{k}$ corresponding to a family of classes in $\CH^1(X_t,1)$ which come from restricting the torus coordinate to $X_t$.

This generalizes well to families of CY $(n-1)$-folds given as level sets in $\Gm^{n}$ of a ``tempered'' Laurent polynomial $\phi$ (cf.~\S\ref{Stemp}).  Here the family of ``cycles'' in $\CH^{n}(X_t,n)$ is given by restricting the torus coordinates to $X_t$,\footnote{That is, we consider the restriction of the diagonal in $(\PP^1)^{n}\times (\PP^1)^{n}$ to $X_t\times (\PP^1\sm \{1\})^{n}$ and ``complete'' it to a higher Chow cycle.} with regulator period $\log(t)+\sum_{k>0}a_k\tfrac{t^k}{k}$; here $a_k$ is the constant term in $\phi^k$, which in the hypergeometric cases is a product of multinomial symbols.  Such expressions suffice for families of elliptic curves (cf.~\S\ref{Sell}), but not for CY 3-folds, even if we can compute more Betti periods of the regulator class.

\subsection*{(ii) Two classes from a single family}
Indeed, the cycles and regulators just described (and studied extensively in \cite{DK,BKV,Ke}) only relate to a single class in $\CH^4(X,4)$ in the context of CY 3-folds.  Thus, we are starting from scratch when it comes to $\CH^3(X,2)$.  Moreover, the rank conjecture for $\CH^4(X,4)$ predicts the existence of \emph{two}, rather than one, independent cycles.  In a way, both issues can be addressed by the choice of hypergeometric indices (together with applying a base-change to the family).  Indeed, the $(2\pi\ay)^m$ ambiguities in our Abel-Jacobi integration schemes, which control the ``types'' of the putative higher cycle families (via $m\longleftrightarrow \CH^m(X,2m-4)$), are determined by the hypergeometric data.

%The Hodge-theoretic avatars of families of higher cycles, as will be made clear in \S\ref{S3}, are certain limiting Hodge-Tate classes over points of the discriminant locus.  The weights of these classes, which determine the ``type'' of the higher cycle, are directly determined by the hypergeometric data.

To state our  $\CH^4(X,4)$ result, consider the (singular) CY 3-fold family $X_t\subset (\PP^1)^4$ given by
$\textstyle 0=1-t\prod_{i=1}^4\frac{(x_i+1)^2}{x_i}$
over $\CC\sm \{0,\tfrac{1}{256}\}$.  The cohomology $\IH^3(X_t)$ contains the hypergeometric VHS with data $(\fa,\fb)=((\tfrac{1}{2})^4,(1)^4)$ and period $\sum_{k\geq 0}\binom{2k}{k}^4 t^k$.
Set $G_k=\sum_{\ell=k+1}^{2k}\tfrac{1}{\ell}$ and $\tilde{G}_k:=8G_k^2-2\sum_{\ell=1}^{2k}\tfrac{1}{\ell^2}+\sum_{\ell=1}^k\tfrac{1}{\ell^2}+\zeta(2).$ Then we have

\medskip

\sc Theorem A \rm (cf.~Thm.~\ref{tK4}). (a) \it For every $t \in (0,1/256)$, there exist two linearly independent higher cycles $\xi_t, \fz_t \in \CH^4(X_t,4)$  such that the coresponding regulator volume admits the hypergeometric expression

$$\det\begin{psmatrix}
\sqrt{t}\sum_{k\geq 0}\frac{\binom{2k}{k}^4}{k+\frac{1}{2}}t^k & \phantom{hh}\frac{\sqrt{t}}{4\pi^2}\sum_{k\geq 0}\binom{2k}{k}^4 t^k\left\{\substack{\frac{4\tilde{G}_k(k+\frac{1}{2})^2-8G_k(k+\frac{1}{2})+1}{(k+\frac{1}{2})^3} \\ +\frac{8G_k(k+\frac{1}{2})-1}{(k+\frac{1}{2})^2}\log(t)+\frac{1}{k+\frac{1}{2}}\log^3(t)}\right\}
\\ \phantom{h}&\phantom{h}\\
\log(t)+\sum_{k>0}\frac{\binom{2k}{k}^4}{k}t^k & \frac{1}{4\pi^2}\left(\substack{\{-8\zeta(3)+\sum_{k>0}\binom{2k}{k}^4\frac{4\tilde{G}_k k^2-8G_k k +1}{k^3}t^k\} \\ +\{4\zeta(2)+\sum_{k>0}\binom{2k}{k}^4\frac{8G_k k-1}{k^2}t^k\}\log(t)\\ +\{\sum_{k>0}\binom{2k}{k}^4\frac{1}{k}t^k\}\log^2(t)\;+\;\frac{1}{6}\log^3(t)} \right)
\end{psmatrix}.$$

\rm (b) \it If, in addition, $t \in \{ 4^{-5},\,4^{-6},\,4^{-7},\,4^{-8}\}$, the cycles
$\xi_t,\fz_t$ extend to cycles in $\CH^4(\mathfrak{X}_t,4)$ of an integral model.\rm 

\medskip

This is used to numerically verify Beilinson's conjecture for these 4 fibers at the end of \S\ref{S6}.  As a conceptual warm-up to this result, in \S\ref{S2cl} we produce a hypergeometric expression for a rank-2 regulator of a number field.

\subsection*{(iii) Classes of non-symbol type}

Theorem A, together with rank 1 cases of $\CH^2(X,0)$ (or ``Birch--Swinnerton-Dyer'') regulators studied in \cite{Go23}, provides evidence in favor of Beilinson's conjecture for $L''(M_t,0)$ and $L'(M_t,2)$, where $M_t$ is a family of hypergeometric motives. To treat the $L'(M_t,1)$ case, we have to construct cycles and compute regulators for $\CH^3(X,2)$.  This is done by taking $X=X_t$ to be a (Hadamard) fiber product of two families of elliptic curves, on which one has pre-existing families of cycles whose cup-product lives on $X$.  (Informally, think $\CH^1(\mathcal{E}^{(1)},0)\cup \CH^2(\mathcal{E}^{(2)},1)$, though this isn't quite correct; see \S\ref{Shad}.) 

We can state the result as follows.  Let $X_t$ denote the CY 3-fold family 
with twisted hypergeometric period $\sum_{k\geq 0} (-1)^k\binom{4k}{2k}\binom{2k}{k}^3 t^k$ constructed in \S\ref{Shad}.  Though there is only one relevant cycle, the Beilinson regulator is still given by a $2\times 2$ determinant as we explain in \S\ref{S4}.  Write $z=2^{10}t$, $\Gamma^{\alpha}_k:=\prod_{i=1}^4\frac{\Gamma(\alpha+k)}{\Gamma(\alpha+k+\fa_i)}$,  $S_{\alpha}(z):=\sum_{k\geq 0}\Gamma^{\alpha}_kz^{-k}$, and $R_{\alpha}(z):=\sum_{k\geq 0}\Gamma^{\alpha}_k\frac{z^{-k}}{k-\frac{1}{2}+\alpha}$ 
(with $\sum_{k>0}$ in $R_{\alpha}$ if $\alpha=\tfrac{1}{2}$), where $z=2^{10}t$.

\medskip

\sc Theorem B \rm (cf.~Prop.~\ref{p10a} \& Thm.~\ref{tK2}). (a) \it For every $t\in (2^{-10},\infty)$, there exists a cycle $Z_t\in \CH^3(X_t,2)$ whose regulator volume has the hypergeometric expression
$$\det\left(\begin{matrix}  4(\log(4z)+4) -\frac{\sqrt{2}}{\pi}R_{\frac{1}{2}}(z) & \frac{4\sqrt{2}}{\pi\sqrt{z}}S_{\frac{1}{2}}(z) \\ -\frac{1}{4\pi}(z^{\frac{1}{4}}R_{\frac{1}{4}}(z)+z^{-\frac{1}{4}}R_{\frac{3}{4}}(z)) & \frac{1}{\pi}(z^{-\frac{1}{4}}S_{\frac{1}{4}}(z)+z^{-\frac{3}{4}}S_{\frac{3}{4}}(z)) \end{matrix}\right).$$

\rm (b) \it This cycle extends to an integral model for every $t$ of the form $n^2/4^k$ \textup{(}$n\in \ZZ_{>0}$, $k\in \ZZ$\textup{)}.\rm

\medskip

\noindent Once again, this leads to numerical verifications of Beilinson.

\subsection*{Why hypergeometric variations?}
According to the principle explained at the beginning, it suffices to construct an ``integration pattern'' --- or more formally, a \emph{normal function} --- mimicking the log function with its $(2 \pi i)\, \ZZ$ ambiguity arising from the Hodge\,(--Tate) structure on $H^1 (\Gm, \ZZ) = \ZZ (-1)$ itself. To do it, we stick with $\Gm$ but introduce cohomology with coefficients in a hypergeometric variation $\mathcal{H}=\mathcal{H}_{\fa,\fb}$ of pure Hodge structure.  The role of the function $\log z=\int 1\tfrac{dz}{z}$ will now be played by the primitive of a hypergeometric period.  By Katz, hypergeometricity is equivalent to $\mathrm{ih}^1 (\Gm, \mathcal{H})=1$, forcing $\IH^1(\Gm,\mathcal{H})$ to be Tate, hence generated by the class of a normal function, and thus (conjecturally) by a family of algebraic cycles.  By passing to a double cover via $z\mapsto z^2\colon \Gm\to \Gm$, this $\IH^1$ becomes a direct sum of Tate classes in possibly different weights, which we can control by tuning the indices $\fa_i, \, \fb_j$.  We demonstrate the computation of these weights, hence the types of putative cycles, in the table of \S\ref{S5} for hypergeometric $(1,1,1,1)$-variations with a maximal unipotent monodromy point.  (Also see Appendix \ref{appB}.)  This shows that hypergeometric variations yield a rich source of examples which can be calibrated to produce cycles of the desired type.

Hypergeometric VHS and their basechanges are also the only ones for which the \emph{Frobenius deformation} $\Phi(s,z)$ defined by \eqref{e2h} can be computed in closed form easily.  This makes it straightforward to compute all the Betti periods of a holomorphic section, and can be generalized to regulator periods.  Moreover, both kinds of periods admit explicit analytic continuation in the hypergeometric setting.

\subsection*{Outline of the paper}

We begin in \S\ref{SAJ} with an informal account of Bloch's higher Chow groups and the explicit Abel-Jacobi maps which send them to Deligne cohomology.  The two simplest examples relevant to this paper (and to the Beilinson conjectures) are $\CH^2(E_t,2)$ of a variant of the Legendre elliptic curve family, and $\CH^1(X_t,1)$ for a family of pairs of points (specializing to quadratic field extensions).  These are worked out in \S\S\ref{Sell}-\ref{SCY0}, followed by an example of rank-two regulators for number fields in \S\ref{S2cl}.

Next, we discuss two approaches to solving inhomogeneous equations of the form $L(\cdot)=t^{1/d}$, where $L$ is a hypergeometric differential operator attached to a family of CY varieties.  The first (in \S\ref{S2}) is by elementary complex analysis, using the Frobenius deformations of \cite{GZ,BV,Ke}, and gives an explicit series solution.  The second is via generalized normal functions on a base-change of the family, which canonically give solutions to inhomogeneous Picard-Fuchs equations by \cite{dAM,GKS}, cf.~\S\ref{S3}.  We also include a new result (extending Thm.~5.1 of \cite{GKS}) about normalizing the lifts of normal functions, which is surprisingly simple and crucial for computing regulator classes in families:

\medskip

\sc Theorem C \rm (cf.~Thm.~\ref{t3a}). \it Suppose we have
\begin{itemize}[leftmargin=0.5cm]
\item a weight $n$ polarized $\QQ$-VHS $\mathcal{M}$ on $U\subset \PP^1_t$, of $\bar{\QQ}$-motivic origin, with all Hodge numbers $h^{k,n-k}=1$ \textup{(}$0\leq k\leq n$\textup{)} and satisfying the remaining assumptions of \S\ref{S2};
\item a holomorphic section $\mu$ of the extended Hodge line bundle $\mathcal{F}_e^n\mathcal{M}_e \cong \mathcal{O}_{\PP^1}(h)$, with corresponding Picard-Fuchs operator $L$ of degree $d$; and
\item a nonzero normal function $\nu\in \mathrm{ANF}_U(\mathcal{M}(p))$ \textup{(}for some $p\in [\tfrac{n+1}{2},n+1]\cap \ZZ$\textup{)}, nonsingular on $\Gm$, arising from a cycle.  \textup{(}Here $\mathrm{ANF}$ stands for ``admissible normal function''.\textup{)}
\end{itemize}
Then there exists a multivalued lift $\tilde{\nu}$ of $\nu$ to $\mathcal{M}$ such that $\nabla_{t\frac{d}{dt}}\tilde{\nu}=\mathsf{Q}_0(2\pi\ay)^n f\mu$, where $\mathsf{Q}_0\in \QQ^{\times}$ and $f:=L\langle \tilde{\nu},\mu\rangle$ is a nonzero polynomial in $\bar{\QQ}[t]$ of degree $\leq d-h$.\rm

\medskip

Since one expects algebraic cycles (both ``classical'' and ``higher'') to underlie these normal functions via Abel-Jacobi maps, the special values of the solutions (at algebraic points in moduli) should be of arithmetic interest.  In \S\ref{S4} we review how the AJ maps are refined into regulator maps at $\QQ$-points, and the relevant cases of the Beilinson conjectures at those points. Turning then to the hypergeometric examples of CY 3-fold type classified %by Doran and Morgan 
in \cite{DM}, 
we explain in \S\ref{S5} how to identify which types of cycles arise --- viz., $K_0$, $K_2$, or $K_4$ classes.

The remainder of the article is focused on two case studies.  In \S\ref{S6} we consider a pair of $K_4$-cycles on the double-cover of the family with hypergeometric data $(\fa,\fb)=((\tfrac{1}{2})^4,(1)^4)$ and use Frobenius deformations to produce an explicit expression for Beilinson's regulator determinant $r(t)$.  In \S\ref{S8} a different approach, via Hadamard products and analytic continuation via Mellin-Barnes integrals, is used to evaluate the regulator $r(t)$ of a family of $K_2$-cycles on double cover of the family with $(\fa,\fb)=((\tfrac{1}{4},\tfrac{1}{2},\tfrac{1}{2},\tfrac{3}{4}),(1)^4)$.  Since this is the first place we know of where Hadamard products are used to construct higher cycles and evaluate AJ maps in families, we include in \S\ref{S7} a more elementary example of how this works.

At the end of sections \S\ref{S6} and \S\ref{S8}, we present numerical evidence that the respective quantities $L''(M_t,0)/r(t)$ and $L'(M_t,0)/r(t)$ are rational for those $t\in \QQ$ for which the cycle extends to an integral model of the given hypergeometric motive $M_t$.  The numerator of each quantity (which we use Magma and Pari to compute) is related to Galois representations and automorphic forms for $\mathrm{Sp}_4$, subject to the standard conjectures on analytic continuation and functional equations for the $L$-functions.  This paper is focused on developing formulas for the denominator, which belongs to the world of periods, algebraic cycles, and Hodge theory.  %In Appendix \ref{appA}, we display our Magma code for checking Beilinson for the $K_2$-cycles in \S\ref{S8}.

To recapitulate, our formulas arise as special cases in a larger study of inhomogeneous differential equations attached to families of generalized algebraic cycles \cite{dAM,MW,DK,GKS}.  While the Beilinson Conjectures are about real regulators, using integral or rational regulators (with holomorphically varying periods) turns out to be the key.  These ``generalized periods'' are attached to higher normal functions, which allow us to bypass the construction of cycles --- except for determining which $t\in \QQ$ should satisfy the conjecture.  In this spirit, we offer a short conjectural computation in Appendix \ref{appB} demonstrating how a  check of Beilinson for $K_2$ might go without cycles in a non-MUM hypergeometric case.

As far as we know, the two case studies in this paper are the first verifications of Beilinson's First Conjecture for families of $\mathrm{Sp}_4$-motives in higher weight.  We hope that the ethos of ``complex analytic functions specializing to artithmetic quantities'' might ultimately point the way to new proofs of the Beilinson Conjectures.

\section{A primer on higher cycles and regulators}\label{Sreg}

In \cite{Bl}, S.~Bloch introduced a geometrization of higher $K$-theory in terms of \emph{higher Chow cycles}, which can be summarized in the Bloch-Grothendieck-Riemann-Roch theorem
$$K_r^{\text{alg}}(X)\otimes \QQ\cong\oplus_p \mathrm{CH}^p(X,r)\otimes \QQ$$
for $X$ smooth quasi-projective.  The higher Chow groups are defined as homology of a complex, $\mathrm{CH}^p(X,r):=H_r\{Z^p(X,\bullet),\partial\}$; here 
\begin{itemize}[leftmargin=0.5cm]
\item $Z^p(X,r)$ are codimension-$p$ cycles on $X\times (\PP^1\setminus\{1\})^r$ meeting $X\times\text{``faces''}$ properly (where ``face'' means that some $\PP^1$-coordinates are set to $0$ or $\infty$); and
\item the differential $\partial$ is given by alternating sum of facet intersections.
\end{itemize}
By the ``normalization'' property they can be thought of as relative Chow groups, viz.
\begin{equation}\label{eqCH}
\mathrm{CH}^p(X,r)\cong \mathrm{CH}^p(X\times \square^r_{\text{rel}}),
\end{equation}
where $\square_{\text{rel}}:=(\PP^1\setminus\{1\},\{0,\infty\})$ is a sort of algebraic loop space (with ``Lefschetz dual'' $\square_{\text{rel}}^{\vee}=(\mathbb{G}_m,\{1\}$).  We will use ``$K_r$'' as a shorthand, with $K_0$ standing for usual cycles and $K_{>0}$ for ``higher'' ones.  Here is a brief ``table of contents'' for the examples involving Calabi-Yau families worked out in this paper:
\begin{itemize}[leftmargin=0.5cm]
\item \S\ref{Sreg} (and \S\ref{S7}): $K_2$ of CY 1-folds;
\item \S\ref{SCY0}:  $K_1$ of CY 0-folds;
\item \S\ref{S6} (and \S\ref{S7}):  $K_4$ of CY 3-folds;
\item \S\ref{S8}:  $K_0$ (and $K_2$) of CY 1-folds and $K_2$ of CY 3-folds.
\end{itemize}

What follows is an extremely informal description of higher cycles and their Abel-Jacobi (integral regulator) maps.  When these vary in families, one gets normal functions, which are abstract Hodge-theoretic objects reviewed in \S\ref{S3}.

\subsection{Abel-Jacobi maps}\label{SAJ}
If $X$ is smooth projective, then \eqref{eqCH} has a generalized Griffiths AJ map to the  generalized intermediate Jacobian
\begin{align*}
J^{p,r}(X)&:=\frac{H^{2p-1}(X\times \bxr^r,\CC)}{F^pH^{2p-1}(X\times \bxr^r,\CC)+H^{2p-1}(X\times\bxr^r,\ZZ(p))}\\ &\cong \tfrac{H^{2p-r-1}(X,\CC)}{F^pH^{2p-r-1}(X,\CC)+H^{2p-r-1}(X,\ZZ(p))}\cong \mathrm{Ext}^1_{\text{MHS}}(\ZZ(0),H^{2p-r-1}(X)(p))
\end{align*}
where the first line equals the second since $H^i(\bxr^r)\cong \ZZ(0)$ for $i=r$ and $0$ otherwise.  Heuristically speaking, if $\partial{\bm \Gamma}=\fz$ in the complex of relative chains $\mathcal{C}_{\bullet}(X\times (\GG_m,\{1\})^r)$, then $$\mathrm{AJ}(\fz)(\omega):=\int_{\bm \Gamma}\omega\wedge \tfrac{dz_1}{z_1}\wedge\cdots\wedge \tfrac{dz_r}{z_r}$$ on test forms $\omega$ in 
$$F^{d_X-p+1}H^{2d_X-2p+r+1}(X,\CC)\cong \left(\frac{H^{2p-r-1}(X,\CC)}{F^pH^{2p-r-1}(X,\CC)}\right)^{\vee}.$$
Using calculus of currents, one then reinterprets this as $$\mathrm{AJ}(\fz)(\omega)=\int_X \omega\wedge \tilde{R}_{\fz},$$ where $R_{\fz}\in D^{2p-r-1}(X)$ is $(2\pi\ay)^{p-r}$ times the Radon transform\footnote{Heuristically, this means:  pull back to $\fz$, push forward to $X$.  Where some justification/moving lemma is needed is in checking that the ``pullback'' is defined as a current.  See \cite{KLi} and references therein.} via $\fz$ of the regulator current $R_r$ defined iteratively\footnote{See \cite{Ke03} for the explicit iterative geometric construction of $\bm{\Gamma}$ which ``explains'' this.} by $$R_r(z_1,\ldots,z_r):=\log(z_1)\tfrac{dz_2}{z_2}\wedge\cdots \wedge \tfrac{dz_r}{z_r}-2\pi\ay \delta_{T_{z_1}}.R_{r-1}(z_2,\ldots,z_r),$$
where $T_f:=f^{-1}(\RR_-)$. The tilde over $R$ means that $R_{\fz}$ itself may not be closed, and you need to replace it by $$\tilde{R}_{\fz}:=R_{\fz}-\Xi+(2\pi\ay)^p\delta_{\Gamma}$$
where $\Gamma\in \mathcal{C}_{2d_X-2p+r+1}(X)$ bounds on the Radon transform (denoted $T_{\fz}$) of $T_{z_1}\cap\cdots \cap T_{z_r}$ and $\Xi \in F^p {D}^{2p-r-1}(X)$ bounds on the Radon transform (denoted $\Omega_{\fz}$) of $(2\pi\ay)^{p-r}\tfrac{dz_1}{z_1}\wedge\cdots \wedge \tfrac{dz_r}{z_r}$.  But $\Xi$ dies against the test forms $\omega$ by type, so can be ignored in the AJ integral.

The justification that this works integrally is in \cite{KLi}, though we don't need that (real regulator is coarser than rational).  There are simplifications in special cases:  to make $\tilde{R}$ closed, we don't need the $\Xi$ current if $p>d_X$ or $p\leq r$, and we don't need $\Gamma$ if $2p-r\notin [0,2d_X]$ (or if $T_{\fz}$ is empty).  For the examples in the remainder of this section, we can ignore both and work with $R_{\fz}$.

Because it is briefly relevant in \S\S\ref{S7}-\ref{S8}, we recall how the Abel-Jacobi map was formalized in \cite{KLe}.  There is a rationally quasi-isomorphic subcomplex of ``$\RR$-proper'' cycles $Z^p_{\RR}(X,-\bullet)\subset Z^p(X,-\bullet)$ on which the above currents are well-defined.  We also have a complex
$$\mathcal{C}_{\mathcal{D}}^{2p+\bullet}(X,\ZZ(p)):=\mathcal{C}^{2p+\bullet}_{\text{sing}}(X;\ZZ(p))\oplus F^p D^{2p+\bullet}(X)\oplus D^{2p+\bullet-1}(X)$$ with the ``cone differential'' $D(T,\Omega,R):=(-\partial T,-d\Omega,dR-\Omega+\gamma)$, whose $k^{\text{th}}$ cohomology is Deligne cohomology $H^{2p+k}_{\mathcal{D}}(X,\ZZ(p))$.  Sending elements $\fz\in Z^p_{\RR}(X,r)$ to 
\begin{equation}\label{eTriples}
(2\pi\ay)^{p-r}((2\pi\ay)^r T_{\fz},\Omega_{\fz},R_{\fz})\in \mathcal{C}_{\mathcal{D}}^{2p-r}(X,\ZZ(p))
\end{equation}
induces a morphism of complexes, hence a map from $\mathrm{CH}^p(X,r)_{\QQ}$ to $H^{2p-r}_{\mathcal{D}}(X,\QQ(p))$.  This is compatible with pushforwards, pullbacks, and products.

\subsection{Two examples}\label{Sell}
Consider first the relative cycle  in $\bxr^3$ given parametrically by $z\longmapsto (1-\tfrac{\ay}{z},1-z,z^4)$, and take $\fz_0\in \mathrm{CH}^2(\QQ(\ay),3)$ to be $\tfrac{1}{4}$ times this.  In this case, $X$ is a point, $2p-n-1=0$, and $R_{\fz_0}\in \CC/\QQ(2)$ is a number given by\footnote{Here $\dlog(1-\tfrac{\ay}{z})\wedge\dlog(1-z)=0$ kills the first term of $R_3$, and the fact that $T_{1-\frac{\ay}{z}}=[0,\ay]$ and $T_{1-z}=[1,+\infty]$ don't intersect kills the last term.}
\begin{align*}
\frac{1}{4(2\pi\ay)}\int_{\PP^1}R_3(1-\tfrac{\ay}{z},1-z,z^4)	
&=-\int_{\PP^1}\delta_{T_{1-\frac{\ay}{z}}}.\log(1-z)\dlog(z^4)\\
&=-\int_0^{\ay}\log(1-z)\frac{dz}{z}=\mathrm{Li}_2(\ay).
\end{align*}
If $\fz:=\fz_0-\overline{\fz_0}$, then we get $\mathrm{AJ}(\fz)=\mathrm{Li}_2(\ay)-\mathrm{Li}_2(-\ay)=2\ay L(\chi_{-4},2)$, in agreement with Borel's theorem.  A similar (but more complicated) computation shows that the cycle $Z\in \mathrm{CH}^2(\QQ(\ay),3)$ parametrized by\footnote{This is a higher Chow cycle; to get a relative cycle in $\bxr^3$ one has to subtract a second component parametrized by $(z,-1,-1)$ from it. This has no effect on the regulator.} $$z\longmapsto (z,f(z),g(z)):=\left(z,-\frac{(z+1)^2}{(z-1)^2},-\frac{(z+i)^2}{(z-i)^2}\right)$$ has $\mathrm{AJ}(Z)=16\ay L(\chi_{-4},2)$.

Next look at the family $\mathcal{E}\to \PP^1$ of elliptic curves given by $$E_t:=\{([X_0{:}X_1],[Y_0{:}Y_1])\mid X_0X_1Y_0Y_1=t(X_1-X_0)^2(Y_1-Y_0)^2\}\subset \PP^1\times\PP^1$$
with singular fibers at $t=0$ ($\mathrm{I}_4$), $t=\tfrac{1}{16}$ ($\mathrm{I}_1$), and $t=\infty$ ($\mathrm{I}_1^*$).  We write the equation as $1-t\phi=0$ for short, with $\phi=(x-1)^2(y-1)^2/xy$. One easily checks that $\xi_t:=\{(e,x(e),y(e))\mid e\in E_t(\CC)\}$ defines a relative cycle in $E_t\times \bxr^2$, and the regulator current $R_t=R_{\xi_t}=\log(x)\dlog(y)-2\pi\ay\log(y)\delta_{T_x}$ is closed for $t\notin [0,\tfrac{1}{16}]\cup\{\infty\}$.  In fact, we have a global higher cycle $\xi=\{x,y\}\in \mathrm{CH}^2(\mathcal{E}\setminus E_0,2)$, of which $\xi_t\in \mathrm{CH}^2(E_t,2)$ are specializations, with integral regulator classes $[R_t]\in H^1(E_{t,\CC}^{\text{an}},\CC/\ZZ(2))$.

For $0<|t|\leq\tfrac{1}{16}$, the ``vanishing cycle at $0$'' $\ve_0$ is well-defined, and we want to compute $\Psi(t):=\int_{\ve_0}R_t$.  Since $\mathrm{Tube}(\ve_0)=\mathbb{T}:=\{|x|=|y|=1\}\subset \PP^1\times\PP^1\setminus E_t$ and $\mathrm{Res}\{t^{-1}-\phi,x,y\}=\{x,y\}|_{E_t}$, one deduces that
\begin{equation}\label{ePsi}
\Psi(t)\equiv \frac{1}{2\pi\ay}\int_{\mathbb{T}}\log(t^{-1}-\phi)\frac{dx}{x}\wedge\frac{dy}{y}=-2\pi\ay\left\{\log(t)+\sum_{m>0} \binom{2m}{m}^2\frac{t^m}{m}\right\},
\end{equation}
mod $\QQ(2)$, where $\binom{2m}{m}^2$ arises as the constant term in $\phi^m$.

What happens at $t=\tfrac{1}{16}$, where $E_{\frac{1}{16}}$ is a nodal rational curve?  The right-hand side of \eqref{ePsi} still converges, and we can compute the answer in a different way, by pulling back $R_{\xi_{\frac{1}{16}}}$ along a normalization.  Noting that $(f(z)-1)^2(g(z)-1)^2=16f(z)g(z)$, we deduce that this is given by $z\mapsto (f(z),g(z))$, and $\ve_0$ pulls back to $T_z$.  But $\int_{T_z}R_2\{f(z),g(z)\}$ is precisely $\mathrm{AJ}(Z)$ above, and so we conclude that $\Psi(\tfrac{1}{16})=16\ay L(\chi_{-4},2)$.  (Intuitively, what is happening here is that the nodal rational curve becomes another copy of $\bxr$, and so we ``go up in $K$-theory'' from $K_2$ to $K_3$.)  Conclude that
\begin{equation}\label{eId}
\log(16)-\sum_{m>0}\frac{\binom{2m}{m}^2}{16^m m}=\frac{8}{\pi}L(\chi_{-4},2).
\end{equation}
Related identities and complete details may be found in \cite[\S6]{DK}.

\subsection{Discussion}

A message common to \cite{DK,BKV,GKS,DKS} is that the use of higher cycles and the regulator calculus produces nontrivial number-theoretic identities with remarkable ease.  In this paper, we get even more power in the hypergeometric case by combining this with Frobenius deformations and differential properties of the regulator.

For example, an easy consequence of \eqref{ePsi} is that $\nabla_D [\tilde{R}_t]=2\pi\ay\mu_t$, where $\mu_t$ is the standard residue form with hypergeometric function as period. This is a special case of \cite[Cor.~4.1]{DK}, of which we will obtain a substantial generalization in Theorem \ref{t3a}.  In principle, this allows us to bypass constructing cycles and computing $R_Z$.

Identities such as \eqref{eId} result from evaluating periods of $R_t$ at points in the discriminant locus by specializing $\xi$ to motivic cohomology of a singular fiber and computing its regulator in another way.  In this paper, we want to specialize cycles like $\xi$ at values of $t\in \QQ$ \emph{not} in the discriminant locus.  If this specialization extends to an integral model $\fx_t$ of $X_t$, then the Beilinson conjectures make a prediction about its regulator.  More precisely, they make a prediction about the \emph{Beilinson regulator}, which (informally) is the determinant of the real regulator map obtained from projecting from $\ZZ$ to $\RR$ coefficients (see \S\ref{S4}).

In the present case, for example, if $t^{-1}$ is a nonzero integer then $\xi_t$ belongs to $\mathrm{CH}^2_{\ZZ}(E_t,2)$, the image of $K_2$ of the integral model in $\mathrm{CH}^2(E_t,2)$.  The prediction is then (taking $t^{-1}>16$) that (i) $\tfrac{1}{2\pi}\text{Im}(\Psi(t))$ (or equivalently, $r(t):=\log(t^{-1})-\sum \binom{2m}{m}^2\tfrac{t^m}{m}$) is a rational multiple of $L'(E_t,0)$ and (ii) $\mathrm{CH}^2_{\ZZ}(E_t,2)$ has rank one.  Since $E_t$ is modular, by Beilinson's results on modular curves \cite{Be2} we know the analogue of (i) for certain $K_2$-classes constructed from Siegel units.  However, since (ii) is completely unknown, the prediction (i) about $\Psi(t)$ does not follow from [op.~cit.].  In fact, a conjecture of Bloch and Kato \cite{BK} (not discussed in this paper) makes the identity of the rational constant rather precise.  Though we will have nothing to say about Bloch-Kato in this paper, numerical evidence for the above elliptic family is presented in the forthcoming work \cite{DGJK}.

The remainder of this paper is concerned with obtaining explicit expressions for the Beilinson regulator in families of higher-dimensional Calabi-Yaus.  However, there are a couple of useful pieces of motivation which we find valuable, and so our next stop is a detour through dimension zero.

\section{Motivation from number theory}\label{S1}

The Beilinson regulator has roots in Griffiths's Abel-Jacobi map (on the geometric side) and Borel's regulator (on the arithmetic side).  A special case of the latter is the Dirichlet regulator of a number field, with the class number formula as a special case of the Beilinson conjectures.  Since the raison d'\^etre of this paper is the interpolation of regulators for families of hypergeometric Calabi-Yau varieties, it is natural to begin at the beginning and ask:  how do you interpolate Dirichlet regulators in a CY family?

\subsection{Quadratic extensions}\label{SCY0}
Consider the quadratic equation 
\begin{equation}\label{equad}
x^2+(2-t^{-1})x+1=0,
\end{equation}
or (equivalently) $1-t\phi(x)=0$ with $\phi(x)=\tfrac{(x+1)^2}{x}$.  It defines a family $X_t=\{x_+(t),x_-(t)\}$ of ``CY 0-folds'' as $t$ runs over $U=\PP^1\setminus \{0,\tfrac{1}{4},\infty\}$.  As our family of cycles, we take $\{x\}|_{X_t}\in \mathrm{CH}^1(X_t\times \square_{\text{rel}})\cong \mathrm{CH}^1(X_t,1)\cong K_1(X_t)$, with integral regulator in $H^0(X_t,\CC/\ZZ(1))$ represented by the 0-current $\log(x)$.  Pairing it with the ``topological cycle'' $[x_+]-[x_-]$ gives
\begin{equation}\label{elog}
r(t):=\log x_+-\log x_-=-2\log x_- = 2\oint_{|x|=1}\log(t^{-1}-\phi)\frac{dx}{x}=-2\log(t)-2\sum_{k>0}\binom{2k}{k}\frac{t^k}{k}
\end{equation}
by Jensen's formula, the Cauchy integral formula, and the fact that the constant term in $\phi^k$ is $\binom{2k}{k}$.  

Actually, this family is hypergeometric:  its ``period'' is the pairing of $[x_+]-[x_-]$ with $\mathrm{Res}_{X_t}(\tfrac{dx/x}{1-t\phi})$, which gives $\sum_{k\geq 0} \binom{2k}{k}t^k=(1-4t)^{-1/2}={}_1F_0(\tfrac{1}{2};4t)$.

For $t=\tfrac{1}{n}$ (with $n>5$), the discriminant of \eqref{equad} is $n(n-4)$.  If we take this discriminant to be squarefree, then $x_-=\tfrac{1}{2}(n-2-\sqrt{n(n-4)})$ is a fundamental unit of $K_t:=\QQ(\sqrt{n(n-4)})$, making $\{x\}|_{X_t}\in K_1(X_t)$ an integral generator.  The Dirichlet class number formula gives
\begin{equation}\label{edir}
\frac{L'(\tilde{H}^0(X_t),0)}{r(t)}=\frac{-\zeta'_{K_t}(0)}{2r(t)}=h_{K_t},
\end{equation}
and the fact that the class number $h_{K_t}$ belongs to $\QQ^{\times}$ (of course, it is an integer) is a special case of Beilinson's Conjecture II.  Put differently, the analytic function $r(t)$ (say, on $(0,\tfrac{1}{4})$) can be thought of as interpolating the values $-\zeta'_{K_t}(0)/2h_{K_t}$ (with $t=\tfrac{1}{n}$ as above).

\subsection{Quintic extensions: two classes from a single family}\label{S2cl}
The last example was not so unlike the $K_2$ example in \S\ref{Sreg}.  In later sections our regulator volumes will be given by $2\times 2$ determinants.  Here is another ``$K_1$ of (non-CY) $0$-fold'' regulator which is analogous to the ``$K_4$ of CY $3$-fold'' one in \S\ref{S6}.

Lang quotes the following example due to Artin in \cite{Lang}.  Consider the field 
$$K_1=\QQ [Y]/ (Y^5-Y+1)  \simeq \QQ [x]/ ((x-1)^5-x) $$
Its discriminant equals the discriminant of the defining polynomial, $\Delta = 19 \cdot 151$, and the Galois group of the decomposition field is $S_5$. The
Minkowski bound for the additive lattice prescribes the existence of an ideal in any given class with norm 	
$$\mathbf{N} \mathfrak{p} \le  \frac{5!}{5^5} \left(\frac{4}{\pi}\right)^{r_2} \Delta^{1/2} \approx 3.34...$$ 

In a given class, take an ideal with the minimal norm. It cannot be 2 or 3 because then the ideal would have to be prime and the residue field extension would be of degree 1, but the defining polynomial has no roots in $\mathbb{F}_2$ or $\mathbb{F}_3$. Thus, the class number is 1 and, in order to compute the leading coefficient of $L(K_1,s)$ at $s=0$, we will need its regulator, which has to be the determinant of a rank 2 matrix,
as $K_1$ has 1 real and 2 pairs of complex embeddings.

For a pair of sets of hypergeometric indices $(\fa,\fb)$ of the same cardinality, normalized to have all $\fa_i,\fb_i\in (0,1]\cap \QQ$, set $\Alpha_i:=\Exp{2\pi\ay\fa_i}$ and $\lambda:=\exp{(\sum_i \psi(\fa_i)-\sum_i \psi(\fb_i))}$,
%$\lambda =\exp \sum_i (\psi (\overline{\alpha_i})-\psi (\overline{\beta_i}))$
where $\psi (x)= \frac{\Gamma '(x)}{\Gamma (x)}$.
%and $\overline{y}$ denotes the unique representative of the class $\,y \!\! \mod \ZZ$ in $(0,1]$: $\overline{y} = 1-\{-y\}$, and put
%\footnote{In general, \eqref{eV*} comes with a gamma structure that is defined to be the set $\boldsymbol\gamma = \{ \sum_{s \in s_0 +\ZZ} \mathbf{\Gamma}(s)\, z^s  \mid s_0 \in \CC \}$ of formal solutions  to ???? and is meant to specialize to a Betti structure when the hypergeometric indices are rational.} 
Writing
\begin{equation}\label{eV*}
\mathbf{\Gamma}(s)\,=\, \mathbf{\Gamma}_{\fa, \fb}(s) \,=\, {\mystrut \prod_{i=1}^n \Gamma(s  - \fa_i + 1)^{-1} \prod_{i=1}^n \Gamma (-s + \fb_i)^{-1}} \quad (s\in\CC),
\end{equation}
%$$\Alpha_i =\Exp{2\pi\ay \alpha_i}, \;\Beta_j =\Exp{2\pi\ay \beta_j},$$  
we define series
\begin{equation}\label{eV*'}
S_{\Alpha_j}(t) = \sum_{l=0}^\infty  \mathbf\Gamma(l + \fa_j)\, \frac{(\lambda t)^{l + \fa_j}}{l + \fa_j}
\end{equation}
and 
\begin{equation}\label{eV**}
S_n(t) = \frac{1}{\mathbf{\Gamma}(0)}\sum_{i=1}^4 {\mathrm{A}}_i^n S_{\Alpha_i}(t).
\end{equation}	

Choose now $$\mathbf{\Gamma}^{(0)} = \mathbf{\Gamma}_{(\frac{1}{10}, \frac{3}{10}, \frac{7}{10}, \frac{9}{10}),    (\frac{1}{4},\frac{1}{2},\frac{3}{4},1)}\;\;\text{and}\;\;\mathbf{\Gamma}^{(1)} = \mathbf{\Gamma}_{(\frac{1}{5},\frac{2}{5},\frac{3}{5},\frac{4}{5}),  (\frac{1}{4},\frac{1}{2},\frac{3}{4},1)},$$ and form the respective $S^{(J)}_n(t), \; J=0,1.$ One then finds the following hypergeometric representation for the $L$--derivative: 
$$ L''(K_1,0) =  \frac{25}{2} \mathop{\mathrm{det}} \mathop{\mathrm{Re}} \begin{pmatrix}
		S^{(0)}_1(256) & S^{(1)}_1(1) \\ 
		
		S^{(0)}_3(256) & S^{(1)}_3(1) 	
	\end{pmatrix}$$
	The reason for this is that the Puiseux series $\exp \left(\frac{1}{5}S^{(0)} \left({256}{t^2}\right)\right)$ is a root of $  (X-1)^5 + {16}{t} (X^3+X^2)$
	while 
	$\exp \left(S^{(1)} (t^2)\right)$ is a root of $(x-1)^5 -t^2x, $ 	
	and these polynomials define isomorphic extensions $K_t$ via  $\phi (x) = -((x-1)^3-tx+t/2)\cdot \frac{2}{t} \, :$
$$ \phi \left(\exp \left(S^{(1)} (t^2)\right)\right) = \exp \left(\frac{1}{5}S^{(0)} \left({256}{t^2}\right)\right).$$
Evaluating the four Puiseux series at $t=1$, we obtain 2 units in 2 embeddings of $K_1$ into $\CC$. 
%[Check!] In general, one has:
%\begin{itemize}			
%\item $$\exp \left(\frac{1}{5}S^{(0)} \left({256}{t^2}\right)\right) \text{ is a root of }  (X-1)^5 + {16}{t} (X^3+X^2)$$ 
%\item $$\exp \left(S^{(1)} (t^2)\right)\text{ is a root of }  (x-1)^5 -t^2x, $$ 	
%\end{itemize}

\section{Motivation from differential equations}\label{S2}

Periods of families of algebraic varieties are governed by Picard-Fuchs equations.  Typically these are expressed in homogeneous form, with inhomogeneous equations then governing Abel-Jacobi periods arising from algebraic cycles on the family (see \S\ref{S3}).  Both kinds of equations play a fundamental role in mirror symmetry and mathematical physics, and it was this context which originally motivated the investigation of cycles on hypergeometric CY-families in this paper (see \S\ref{S5}).

The comparison between Frobenius and Betti solutions to a Picard-Fuchs equation in a neighborhood of a maximal unipotent monodromy point has been the subject of recent study in \cite{GZ,BV,Ke}.  It is closely related to both the Gamma conjectures and limiting mixed Hodge structures.  In the hypergeometric case, as we review in the examples below, pretty much everything can be put in closed form.  The long paragraph which follows introduces the Hodge-theoretic setting that will be in effect for the remainder of the paper.

We begin with an irreducible motivic polarized\footnote{We write $\langle\cdot,\cdot\rangle$ for the polarization pairing $\M\times \M\to \co$, $\int_{\cdot}(\cdot)$ for the duality pairing $\MM^{\vee}\times \M\to \co$ (as well as $\VV^{\vee}\times \V\to \co$), and $D=\nabla_D$ for the Gauss-Manin connection.} $\QQ$-VHS $\M$ over $U=\PP^1\setminus \Sigma$ (with coordinate $z$), subject to the following assumptions and notations:
\begin{itemize}[leftmargin=0.5cm]
\item the geometric origin of $\M$ is in a family defined over $\bar{\QQ}$;
\item $\M$ has weight $n$ with all Hodge numbers $h^{k,n-k}=1$ ($0\leq k\leq n$);
\item $\Sigma=\{0,c=c_1,c_2,\ldots,c_d,\infty\}$ with $|c|\leq |c_i|$ ($\forall i$);
\item monodromies of the underlying $\QQ$-local system $\MM$ and its dual are written $T_{\sigma}=T_{\sigma}^{\text{ss}}e^{N_{\sigma}}$ for each $\sigma\in \Sigma$;
\item $\MM$ has maximal unipotent monodromy at $0$, with dual ``vanishing cycle'' $\ve_0$, viz.~$(\MM_p^{\vee})^{T_0}=\QQ\langle \ve_0\rangle$ (for a fixed $p\in U$ close to $0$);
\item $\MM$ has conifold (i.e.~Picard-Lefschetz) monodromy at each $c_i$, and for $i=1$ we write $(T_c-I)\MM_p^{\vee}=\QQ\langle \delta\rangle$;
\item $\mu\in \Gamma(\PP^1,\F_e^n\M_e)$ is a holomorphic section of the canonically extended Hodge line, nonvanishing on $\PP^1\setminus\{\infty\}$, normalized so that $\lim_{z\to 0}\int_{\ve_0}\mu=1$; and
\item $L=\sum_{j=0}^d z^jP_j(D)=\sum_{i=0}^{n+1}q_{n+1-i}(z)D^i\in \CC[z,D]$ (where $D=z\tfrac{d}{dz}$) is an irreducible operator with $L\mu=0$, $\gcd(\{q_i\})=1$, $q_0(0)=1$.  Note that our assumptions (that $\deg(L)=d$ and $\M$ has arithmetic origin) imply that the Frobenius dual $L^{\dagger}=L$ and $L\in \bar{\QQ}[z,D]$.
\end{itemize}
By \cite[Lem.~4.2]{Ke}, we may uniquely extend $\ve_0$ to a basis of $\MM_p^{\vee}$ by $\{\ve_1,\ldots,\ve_n\}\subset (\MM_p^{\vee})^{T_c}$ with $N_0\ve_m=\ve_{m-1}$.  The resulting \emph{Betti periods} $\e_m:=\int_{\ve_m}\mu$ form a (multivalued) $\CC$-basis of solutions for $L(\cdot)=0$, with $\QQ$-monodromies. It is important that we have imposed conditions sufficient to make $\mu$, $L$, and $\{\ve_m\}$ unique; otherwise explicit determinations of (say) the function $\alpha(s)$ below would have no meaning.

Presently we are interested in solutions to the multivalued inhomogeneous equation
\begin{equation}\label{e2a}
L(\cdot)=z^{1/r}	
\end{equation}
which arose in the context of open mirror symmetry nearly two decades ago \cite{Wa}.  In this section we give a purely Hodge-theoretic solution, for which the following result is instrumental:

\begin{prop}[\cite{BV,Ke}]\label{p2a}
The restriction $\M|_{\Delta_0^{\times}}$ to a punctured disk about $z=0$ has a unique pro-unipotent extension to a VMHS $\V\to \Delta^{\times}$ satisfying:
\begin{itemize}[leftmargin=0.5cm]
\item $\V$ has nonzero Hodge-Deligne numbers $h^{-\ell,-\ell}=1$ \textup{(}$\ell>0$\textup{)} and $h^{k,n-k}=1$ \textup{(}$0\leq k\leq n$\textup{)};
\item $\V$ has a surjective morphism to $\M$, and a section $\Omega\in \Gamma(\PP^1,\F^n_e\V_e)$ mapping to $\mu$, with $D^{\infty}L\Omega=0$;
\item the underlying $\QQ$-local system $\VV$ is closed under $T_c$; and
\item $\VV_p^{\vee}$ has a basis $\{\ve_m\}_{m\geq 0}$ with $T_c\ve_m=\ve_m$ and $T_0\ve_m=\ve_{m-1}$ for $m>0$.
\end{itemize}
\end{prop}

We encode the Betti periods $\e_m:=\int_{\ve_m}\Omega$ of $\V$ in the generating series
\begin{equation}\label{e2b}
\begin{split}
\bE(s,z)&:=\phantom{:}\textstyle\sum_{m\geq 0}(2\pi\ay)^m\e_m(z)s^m \\
&\phantom{:}=:\,\textstyle\sum_{k\geq 0}a_k(s)z^{k+s},
\end{split}
\end{equation}
which satisfies
\begin{equation}\label{e2c}
\left\{
\begin{aligned}
(T_c-I)\bE(s,z)&=\textstyle\int_{\delta}\mu \\
\text{and }\;T_0\bE(s,z)&=e^{2\pi\ay s}\bE(s,z)
\end{aligned}
\right.	
\end{equation}
by the last bullet of Proposition \ref{p2a}.  By \cite[Thm.~6.3]{Ke}, the power-series coefficients of $\alpha(s)=1+\sum_m \alpha_m s^m$ defined by
\begin{equation}\label{e2d}
L\bE(s,z)=:s^{n+1}\alpha(s)z^s
\end{equation}
give (the first column of) the period matrix of the limiting MHS $\psi_z \V$ at the origin.

\begin{example}\label{ex2a}
The hypergeometric case can (given the assumptions already made) be characterized equivalently by the conditions
\begin{itemize}[leftmargin=0.5cm]
\item $d=1$ (and $\Sigma=\{0,1,\infty\}$)
\end{itemize}
or
\begin{itemize}[leftmargin=0.5cm]
\item $L=D^{n+1}-z\prod_{j=1}^{n+1}(D+\fa_j)$ (where $\{\fa_j\}=\{1-\fa_j\}$).
\end{itemize}
The $\{e^{2\pi\ay\fa_j}\}$ are then the eigenvalues of $T_{\infty}$; repeated eigenvalues produce quasi-unipotent ``$N_{\infty}$-strings''.  The holomorphic period is 
\begin{equation}\label{e2e}
\e_0(z)=\sum_{k\geq 0}a_k(0)z^k=\sum_{k\geq 0}\prod_{j=1}^{n+1}\frac{\Gamma(k+\fa_j)}{\Gamma(\fa_j)\Gamma(k+1)}z^k=\sum_{k\geq 0}\frac{\prod_{j=1}^{n+1}[\fa_j]_k}{(k!)^{n+1}}z^k,
\end{equation}
while by \cite[Ex.~6.8]{Ke} we have
\begin{equation}\label{e2f}
a_k(s)=\prod_{j=1}^{n+1}\frac{\Gamma(k+s+\fa_j)}{\Gamma(\fa_j)\Gamma(k+s+1)}\;\;\text{ and }\;\;\alpha(s)=\prod_{j=1}^{n+1}\frac{\Gamma(s+\fa_j)}{\Gamma(\fa_j)\Gamma(s+1)}.
\end{equation}
From the form of $a_k(s)$, we see that \emph{the classical ``Frobenius method''\footnote{That is, the standard trick of differentiating a series $\sum_k a_k z^k$ with respect to the summation parameter $k$, provided a reasonable interpolation of the $\{a_k\}$ is available.  See \cite[Appendix A]{Ke}.} gives Betti periods} in the hypergeometric case.
\end{example}

Next define formally
\begin{equation}\label{e2g}
\frac{\bE(s,z)}{\alpha(s)}=:\Phi(s,z)=:\sum_{m\geq 0}\phi_m(z)s^m=\sum_{k\geq 0}\frac{a_k(s)}{\alpha(s)}z^{k+s},
\end{equation}
the generating series of \emph{Frobenius periods}.\footnote{It turns out that all of our ``formal'' functions of $(s,z)$ are actually analytic on $\CC\times \widetilde{U}^{\text{un}}$, cf.~\cite[\S5]{Ke}.} One can uniquely characterize the periods (as solutions to $L(\cdot)=0$) by the asymptotic property $\phi_m(z)-\tfrac{1}{m!}\log^mz\to 0$ with $z\to 0$; the equations
\begin{equation}\label{e2h}
\left\{
\begin{aligned}
L\Phi(s,z)&=s^{n+1}z^s\\
T_0\Phi(s,z)&=e^{2\pi\ay s}\Phi(s,z),
\end{aligned}
\right.
\end{equation}
uniquely characterize the generating series \cite{GZ,BV}.  Now by the second line of \eqref{e2h}, we may consider the function
\begin{equation}\label{e2i}
\W_r(z):=\textstyle r^{n+1}\Phi(\tfrac{1}{r},z)
\end{equation}
as an analytic function of $z^{1/r}$ in a neighborhood of the origin.

\begin{prop}\label{p2b}
$\W_r$ is the unique solution of \eqref{e2a} with $\W_r(0)=0$.
\end{prop}
\begin{proof}
That $\W_r$ solves \eqref{e2a} is immediate from the first line of \eqref{e2h}.  For uniqueness, the general local analytic solution is clearly $\W_r(z)+C\e_0(z)$ ($C\in \CC$), which has value $C$ at $0$.
\end{proof}

\begin{example}\label{ex2b}
In the hypergeometric case, we have 
\begin{equation}\label{e2j}
\W_r(z)=\sum_{k\geq 0}\prod_{j=1}^{n+1}\frac{[\fa_j+\frac{1}{r}]_k}{[\frac{1}{r}]_{k+1}}z^{k+\frac{1}{r}}.
\end{equation}
\end{example}

%%%%%%%%%%%%%%%%
%%%%%%%%%%%%%%%%

\section{Higher normal functions}\label{S3}

Just as variations of Hodge structure (or equivalently, period maps) are attached to families of smooth projective varieties, normal functions are the abstract Hodge-theoretic gadget associated to families of algebraic cycles thereon.  In this section we review several definitions and properties needed in the remainder of the paper.

Let $\H\to U$ be a VHS of pure weight $w<0$.  It is convenient to assume that $\H$ has trivial fixed part; that is, the underlying $\QQ$-local system $\HH$ has $H^0(U,\HH)=\{0\}$.

\begin{defn}
A \emph{normal function} on $U$ with values in $\H$ is
\begin{itemize}[leftmargin=0.5cm]
\item a holomorphic, horizontal section $\nu$ of $\J(\H):=\H/(\F^0\H+\HH)$.
\end{itemize}
or equivalently
\begin{itemize}[leftmargin=0.5cm]
\item a variation of MHS $\E_{\nu}$ which is an extension of the form $$0\to \H\to \E_{\nu}\to \QQ(0)\to 0.$$
\end{itemize}
\end{defn}
One recovers the (global) section from the extension by taking differences of local lifts of $1\in \QQ(0)$ to $\F^0\E_{\nu}$ and to $\EE_{\nu}$ ($=$ the underlying $\QQ$-local system).  A normal function is called \emph{classical} if $w=-1$ and \emph{higher} if $w<-1$, since the first type arises from families of algebraic cycles in the ``classical'' sense (i.e.~$K_0^{\text{alg}}$) while higher normal functions arise from families of higher Chow groups (viz., $K_{-w-1}^{\text{alg}}$).

It will be useful to adopt the shorthand
$$\mathrm{Hg}^p(V):=\mathrm{Hom}_{\text{MHS}}(\QQ(-p),V)\cong\mathrm{Hom}_{\text{MHS}}(\QQ,V(p))$$
for Hodge $(p,p)$ classes in a MHS $V$.  In \S\ref{S4} we shall also make use of 
$$J^p(V):=\mathrm{Ext}^1_{\text{MHS}}(\QQ(-p),V)\cong W_{2p}V_{\CC}/(F^pW_{2p}V_{\CC}+V_{\QQ}).$$

\begin{defn}
A normal function on $U$ is \emph{admissible} (with respect to $\PP^1$) if the limit MHSs\footnote{This is a convenient abuse of notation; it stands for $\psi_{z-\sigma}$ (for $\sigma\neq \infty$) and $\psi_{\frac{1}{z}}$ (for $\sigma=\infty$), i.e.~the nearby-cycles space with the limit MHS layered onto it.} $\{\psi_{\sigma}\E_{\nu}\}_{\sigma\in \Sigma}$ exist; we write $\nu\in \mathrm{ANF}_U(\H)$.  In this case $\nu$ acquires \emph{singularity} invariants\footnote{Here $(W)_{T}$ denotes the $T$-coinvariants $W/(T-I)W$. For $T=e^N$ unipotent, this is the same as $W/NW$.}
\begin{equation}\label{e30}
\text{sing}_{\sigma}(\nu)\in \mathrm{Hg}^0((\psi_{\sigma}\H)_{T_{\sigma}}(-1))
\end{equation}
by applying $N_{\sigma}$ to local lifts of $1$ to the canonical extensions of $\E_{\nu}$, cf.~\cite[\S2.11]{KP}.  For any $U_0\supseteq U$, we set
\begin{equation}\label{e3a}
\mathrm{ANF}_{U_0}(\H):=\{\nu\in \mathrm{ANF}_U(\H)\mid \text{sing}_{\sigma}(\nu)=0\;\forall \sigma\in U_0\cap \Sigma\};
\end{equation}
that is, a $\nu\in \mathrm{ANF}_{U_0}(\H)$ is allowed to be singular only on $\Sigma_0:=\PP^1\setminus U_0$.
\end{defn}

\begin{properties}
We briefly describe some results that aid in computing the abelian groups \eqref{e3a}.  Taking the class of a normal function yields identifications\footnote{To be clear, $\IH^1(U_0,\H)$ is a (Zucker/Saito) mixed Hodge structure with underlying $\QQ$-vector space $\IH^1(U_0,\HH)=H^1(U_0,(j_0)_*\HH)$ where $j_0\colon U\to U_0$ is the inclusion.}
\begin{equation}\label{e3b}
[\,\cdot\,]\colon \mathrm{ANF}_{U_0}(\H)\overset{\cong}{\to}\mathrm{Hg}^0(\IH^1(U_0,\H))
\end{equation}
with Hodge classes in ``partially parabolic'' cohomology, see \cite[\S4]{GKS}.  Since the parabolic cohomology $\IH^1(\PP^1,\H)$ is pure of weight $-w$, this implies at once that
\begin{equation}\label{e3c}
w<-1\;\;\implies\;\;\mathrm{ANF}_{\PP^1}(\H)=\{0\}
\end{equation}
so that \emph{higher} normal functions are defined by their singularities.  On the other hand, the triviality of RHS\eqref{e30} in the classical case yields
\begin{equation}\label{e3d}
w=-1\;\;\implies\;\;\mathrm{ANF}_U(\H)=\mathrm{ANF}_{\PP^1}(\H).
\end{equation}
For computing rank of $\IH^1(U_0,\H)$, we can use Euler-Poincare:  denoting the ``drop'' at $\sigma$ by $\delta_{\sigma}:=\mathrm{rk}(T_{\sigma}-I)$,
\begin{equation}\label{e3e}
\textstyle\mathrm{ih}^1(U_0,\H):=\sum_{\sigma\in U_0\cap \Sigma}\delta_{\sigma}-(\mathrm{rk}\,\H)\cdot \chi_{U_0}.
\end{equation}
Finally, one has exact sequences
\begin{equation}\label{e3f}
0\to \IH^1(\PP^1,\H)\to \IH^1(U_0,\H)\to \oplus_{\sigma\in \Sigma_0}(\psi_{\sigma}\H)_{T_\sigma}(-1)\to 0
\end{equation}
that aid in computing $\IH^1(U_0,\H)$ as a MHS.
\end{properties}

\subsection{Normal functions and differential equations}

Now take $\H=\M(p)$, and suppose there is a \emph{nonzero} 
\begin{equation}\label{e3g}
\nu\in \mathrm{ANF}_{\GG_m}(\M(p))\cong \mathrm{Hg}^p(\IH^1(\GG_m,\M))
\end{equation}
for some $p\in [\tfrac{n+1}{2},n+1]\cap \ZZ$.  Let $\tilde{\nu}:=\nu_{\QQ}-\nu_F$ be a multivalued lift (from $\J(\M(p))=\M/(\F^p\M+\MM(p))$) to $\M$; and consider the multivalued holomorphic function
\begin{equation}\label{e3h}
V:=\langle \tilde{\nu}(t),\mu(t)\rangle.
\end{equation}
Writing $\F_e^n\M_e\cong \co_{\PP^1}(h)$, the differential equations satisfied by normal functions are described by

\begin{prop}[{\cite[Thm.~5.1]{GKS}}]\label{p3a}
The ``inhomogeneous term'' 
\begin{equation}\label{eINHOM}
LV=:f
\end{equation}
is a nonzero polynomial of degree $\leq d-h$.  If $\textup{sing}_0(\nu)=0$, then $t\mid f$.
\end{prop}

We supplement this with a new result which drives home the importance of the inhomogeneity \eqref{eINHOM}.  Under our running assumptions (cf.~the beginning of \S\ref{S2}), note that by \cite[Prop.~7.1]{Ke} the Yukawa coupling $Y:=\langle \mu,D^n\mu\rangle$ satisfies \begin{equation}\label{e1}
Y=q_0^{-1}(2\pi\ay)^{-n}\mathsf{Q}_0^{-1},
\end{equation}
where $\mathsf{Q}_0=\langle \varepsilon_0,\varepsilon_n\rangle\in \QQ^{\times}$. 

\begin{thm}\label{t3a}
There exists a choice of lift $\tilde{\nu}$ \textup{(}unique up to $\MM(p)$\textup{)} such that 
\begin{equation}\label{eINHOM2}
D\tilde{\nu}=\mathsf{Q}_0(2\pi\ay)^nf\mu,
\end{equation}  
with $f$ as in \eqref{eINHOM}.  Moreover, if $\nu$ comes from a cycle as in \eqref{e4b} below, then $f\in \qb[t]$.\footnote{The statement on field of definition of $g$ is very similar to Defn.~3.1 and Thm.~4.1 of \cite{dAM}.  However, their claim (which replaces $\mu$ in the definition of $V$ by $\omega=(2\pi\ay)^n\mu$) is not proved and is incorrect as stated, as can be seen from the examples in \cite{GKS}.}
\end{thm}

\begin{proof}
The first step is to show that the lift can be chosen to make $D\tilde{\nu}$ a (multivalued) section of $\F^n\M$.  We only need to do this locally.  By Griffiths transversality, for an arbitrary lift --- call this $\tilde{\nu}_{p-1}$ --- we know that $D\tilde{\nu}_{p-1}$ is a section of $\F^{p-1}\M$. Now argue inductively:  if $\tilde{\nu}_k$ is a lift with $D\tilde{\nu}_k$ in $\F^k$ (and $k\leq n$), then we can use surjectivity of $\bar{\nabla}_D\colon \gr_{\F}^{k+1}\M\to \gr_{\F}^k\M$ to produce a section $\eta_{k+1}$ of $\F^{k+1}\M$ such that applying $D$ to $\tilde{\nu}_{k+1}:=\tilde{\nu}_k+\eta_{k+1}$ lands in $\F^{k+1}$.

The uniqueness statement comes from the irreducibility of the variation:  there is no flat section of $\F^p\M$, even locally.

Now write $D\tilde{\nu}=F\mu$, where \emph{a priori} $F$ could be $C^{\infty}$ and multivalued.  Observe that by type, $\langle \mu,D^k\mu\rangle=0$ for $k<n$. Since $L$ has degree $n+1$ and kills $\mu$, clearly
\begin{equation}\label{e2}
f=LV=L\langle\tilde{\nu},\mu\rangle=q_0\langle D\tilde{\nu},D^n\mu\rangle=q_0 F\langle \mu,D^n\mu\rangle=\mathsf{Q}_0^{-1}(2\pi\ay)^{-n}F
\end{equation}
follows from \eqref{e1}.

Turning to the field of definition, recall that $\mathcal{M}$ is a subvariation of the $n^{\text{th}}$ fiberwise cohomology of a family, say $\cx_U\to U$. If there exists a cycle $\fz\in \mathrm{CH}^p(\cx_U,2p-n-1)$ motivating $\nu$, then by spreading out and specializing we may assume that $\fz$ is defined over $\qb$.  We then have (on $\cx_U$) the standard regulator currents $R_{\fz}$ and $\Omega_{\fz}$, and the chain $T_{\fz}$, with $d[R_{\fz}]=\Omega_{\fz}-(2\pi\ay)^{2p-n-1}\delta_{T_{\fz}}$, as defined in \cite{KLM,KLe}.  The Deligne cycle class of $\fz$ is represented by the triple $((2\pi\ay)^{p}T_{\fz},(2\pi\ay)^{n+1-p}\Omega_{\fz},(2\pi\ay)^{n+1-p}R_{\fz})$, in which the fiberwise pullbacks of the first two terms are cohomologically trivial.  Moreover, since $(2\pi\ay)^{n+1-p}\Omega_{\fz}$ is constructed from products of $\mathrm{dlog}$ and $(2\pi\ay)^c\delta_{\mathcal{Y}}$ currents ($\mathcal{Y}\subset \cx_U$ of codim.~$c$), from functions and subvarieties involved in the cycle, it is $\bar{\QQ}$-de Rham.  Clearly there exists a $\qb$-dR $F^p\mathcal{D}^n$-current $\Xi$ which kills the 0th Leray graded part of its class.  Locally (or over the universal cover) there exists a chain $\Gamma$ bounding on $T_{\fz}$.  So the fiberwise pullbacks of $\tilde{R}:=(2\pi\ay)^{n+1-p}R_{\fz}-\Xi+(2\pi\ay)^p\delta_{\Gamma}$ are closed, hence represent the class of a preliminary lift $\tilde{\nu}_{p-1}$, with $D\tilde{\nu}_{p-1}$ represented by $d[\tilde{R}]=(2\pi\ay)^{n+1-p}\Omega_{\fz}-d[\Xi]$, which is $\qb$-dR. The further inductive steps above can be done in such a way as to preserve this property.

The upshot is that $D\tilde{\nu}$ is $\qb$-dR, $\mu$ hence $D^n\mu$ is $\QQ(-n)$-dR, and so their fiberwise pairing is $\qb(-n)$-dR as a section of $\H^{2n}$.  Integrating over fibers must therefore give values in $\qb$ on fibers defined over $\qb$, making $f\in \qb[t]$ in view of \eqref{e2}.
\end{proof}

\subsection{The lowest degree cases} 
Our assumptions on $\M$ together with \eqref{e3e} give immediately that $\mathrm{ih}^1(\GG_m,\M)=d$.  This makes analysis of the low-degree cases particularly simple.  In the next two examples we assume that $h=1$ (or equivalently that $F^{n+1}\IH^1(\PP^1,\M)=\{0\}$) and that the normal functions come from cycles.

\begin{example}[{$d=1$}]\label{ex3a}
$\mathrm{ih}^1(\GG_m,\M)=1$ and $(\psi_0\M)_{T_0}(-1)\cong \QQ({-}n{-}1)$ $\implies$ $\IH^1(\GG_m,\M)=\IH^1(\PP^1\setminus\{0\},\M)=\QQ({-}n{-}1)$ $\implies$ $p=n+1$ and $\nu$ is singular at $0$ $\implies$ $LV=c_0\in \bar{\QQ}$.
\end{example}

\begin{example}[{$d=2$}]\label{ex3b}
$\mathrm{ih}^1(\GG_m,\M)=2$ $\implies$ \emph{either}
\begin{itemize}[leftmargin=0.8cm]
\item [\textbf{(i)}] $0\to \IH^1(\PP^1,\M)\to \IH^1(\GG_m,\M)\to (\psi_0\M)_{T_0}(-1)\to 0$ presents $\IH^1(\GG_m,\M)$ as a (possibly nontrivial) extension of $\QQ({-}n{-}1)$ by $\QQ(-\tfrac{n+1}{2})$; \textit{or}
\item [\textbf{(ii)}] $\IH^1(\GG_m,\M)\cong (\psi_0\M)_{T_0}(-1)\oplus (\psi_{\infty}\M)_{T_{\infty}}(-1)$ is the direct sum of $\QQ({-}n{-}1)$ and $\QQ(-p)$ for some $p\in [\tfrac{n+1}{2},n+1]\cap \ZZ$.
\end{itemize}
To decide between (i) and (ii), we simply check whether there are any monodromy-coinvariant classes at $z=\infty$.  The extension in (i) is trivial if $\M$ arises from base-change of a hypergeometric VHS via $z\mapsto z^2$, or from a tempered Laurent polynomial (see \S\ref{S4}).  

In case (ii), or in case (i) with trivial extension, we have two normal functions:  $\nu_0\in \mathrm{ANF}(\M(n+1))$ (singular at $0$); and $\nu_{\infty}\in \mathrm{ANF}(\M(p))$ (nonsingular or singular at $\infty$). Hence the corresponding multivalued functions satisfy $LV_0=c_0$ and $LV_{\infty}=c_{\infty}z$, with $c_0,c_{\infty}\in \bar{\QQ}$.
\end{example}

%%%%%%%%%%%%%%%%%%%%%
%%%%%%%%%%%%%%%%%%%%%

\section{Regulator maps and Beilinson's first conjecture}\label{S4}

Let $\X\to \PP^1$ be a projective family of Calabi-Yau $n$-folds defined over $\bar{\QQ}$, smooth over $U$, with a subvariation $\M\subset \H^n_{\X_U/U}$ as in \S\ref{S2}.  Suppose $\nu\in \mathrm{ANF}_{U_0}(\M(a))$  is a nonzero normal function, with class
\begin{equation}\label{e4a}
[\nu]\in \mathrm{Hg}\left(\IH^1(U_0,\M(a))\right)\subset \mathrm{Hg}\left(H^{n+1}(\X_{U_0,\CC}^{\text{an}},\QQ(a))\right).
\end{equation}
According to the Beilinson-Hodge Conjecture (BHC) there should exist a generalized algebraic cycle\footnote{The subscript $\QQ$ is a shorthand for tensoring with $\QQ$, the middle term is Bloch's higher Chow group, and $K_b^{(a)}$ means the $a^{\text{th}}$ $\gamma$-graded piece of $K_b^{\text{alg}}$. The isomorphisms in \eqref{e4b} remain valid replacing $\X_{U_0}$ by any smooth quasi-projective variety (over any field).}
\begin{equation}\label{e4b}
\mathfrak{Z}\in H_{\mathrm{Mot}}^{n+1}(\X_{U_0},\QQ(a))\cong \mathrm{CH}^a(\X_{U_0},2a{-}n{-}1)_{\QQ}\cong K^{(a)}_{2a-n-1}(\X_{U_0})_{\QQ}
\end{equation}
with $cl(\mathfrak{Z})=[\nu]$.  Put differently, \emph{Hodge-Tate classes in $\IH^1(U,\mathcal{M})$ are the Hodge-theoretic avatars of families of higher cycles}; and the weights of these classes control the $K$-theoretic type of the putative cycle.

Here we remind the reader that, for any smooth quasi-projective $Y$ defined over a subfield of $\CC$, there is an \emph{absolute Hodge cycle-class map}
\begin{equation}\label{e4c}
c_{\text{AH}}\colon H^q_{\mathrm{Mot}}(Y,\QQ(p))\to H^q_{\text{AH}}(Y_{\CC}^{\text{an}},\QQ(p))
\end{equation}
whose composition with the second nontrivial arrow in
\begin{equation}\label{e4d}
0\to J^p(H^{q-1}(Y_{\CC}^{\text{an}}))\to H^q_{\mathrm{AH}}(Y_{\CC}^{\text{an}},\QQ(p))\to \mathrm{Hg}^p(H^q(Y_{\CC}^{\text{an}}))\to 0
\end{equation}
recovers $cl$.  When $cl$ of a cycle $Z$ is zero, its image under $c_{\text{AH}}$ belongs to the left-hand term of \eqref{e4d} and is called its \emph{Abel-Jacobi} (or \emph{rational regulator}) class $\mathrm{AJ}_{Y}(Z)$.  Explicit formulas for these classes are formulated for higher Chow cycles in \cite{KLM}, and were informally recalled above in \S\ref{Sreg}.

Since the fiberwise restrictions of $[\nu]$ are zero, the regulator currents of [op.~cit.] compute $\mathrm{AJ}_{X_z}(\mathfrak{Z}_z)\in J^a(H^n(X_{z,\CC}^{\text{an}}))$ for each $z\in U$, recovering $\nu=:\nu_{\mathfrak{Z}}\in \J(\M(a))$ as a section of the generalized Jacobian bundle.  We are interested in what this says about the ``values'' of $\nu$ at (say) $z\in \QQ$.  To simplify the story, we will assume henceforth that $a>\tfrac{n}{2}+1$ so that we are not dealing with $K_0^{\text{alg}}$ or $K_1^{\text{alg}}$ (thereby avoiding Beilinson Conjectures II and III).

\subsection{Real regulators}

So let $X$ be a smooth projective $n$-fold over $\QQ$, and $a\in [\tfrac{n+3}{2},n+1]\cap \ZZ$.  There are then two simplifications:  absolute Hodge becomes Deligne cohomology; and $\mathrm{Hg}^a(H^{n+1}(\xan))=\{0\}$ since $b:=2a-n-1\geq 2$ is nonzero.  Thus $J^a(H^n(\xan))=H_{\D}^{n+1}(\xan,\QQ(a))$, and we define the \emph{real regulator} to be the composition
\begin{equation}\label{e4e}
r_X\colon K_b^{(a)}(X)_{\QQ}\overset{\mathrm{AJ}_X}{\longrightarrow} H_{\D}^{n+1}(\xan,\QQ(a))\overset{\pi_{\RR}}{\twoheadrightarrow} H_{\D}^{n+1}(\xan,\RR(a)).
\end{equation}
Explicitly, the right hand map takes the form
\begin{equation}\label{e4f}
\frac{H^n(\xan,\CC)}{F^a H^n(\xan,\CC)+H^n(\xan,\QQ(a))}\twoheadrightarrow \frac{H^n(\xan,\RR(a-1))}{F^a+\overline{F^a}},
\end{equation}
using $\CC/\QQ(a)\twoheadrightarrow \CC/\RR(a)=\RR(a{-}1)$.\footnote{On the level of periods $P=R+\ay I$ of the regulator classes, this is nothing more than taking $R$ (resp.~$\ay I$) if $a$ is odd (resp.~even).} Since the cycles coming from $K_b^{(a)}(X)$ are defined over $\QQ$, it turns out that their image under $r_X$ is invariant under the involution $\DR$ given by composing $\mathsf{F}_{\infty}^*$ (conjugating points of $X_{\CC}^{\text{an}}$) with conjugation on the coefficients $\RR(a)=(2\pi\ay)^a\RR$.

Finally, we need to restrict our domain to the cycles which extend to an integral model $\fx$ of $X$ and $\RR$-linearly extend to obtain the Beilinson regulator
\begin{equation}\label{e4g}
\rbe_X\colon K_b^{(a)}(\fx)_{\RR}\to H_{\D}^{n+1}(\xan,\RR(a))^{\DR}\cong \frac{H^n(X_{\CC}^{\text{an}},\RR(a-1))^{\DR}}{F^aH^n_{\mathrm{dR}}(X_{\RR})},
\end{equation}
which is a morphism of real vector spaces.\footnote{Algebraic de Rham cohomology $F^aH^n_{\mathrm{dR}}(X_{\RR})$ clearly embeds in Betti cohomology $H^n(X_{\CC}^{\text{an}},\CC)^{\DR}$, and then we pass to the quotient by $H^n(X_{\CC}^{\text{an}},\RR(a))^{\DR}$. This composite is still an embedding as the first and last groups intersect in $\{0\}$.  This explains the group on the right-hand side of \eqref{e4g}.}  Moreover, the determinant of each side has a natural $\QQ$-structure:  on the left, this is $\det K_b^{(a)}(\fx)_{\QQ}$; and on the right, $\det H^n(\xan,\QQ(a-1))^{\DR}/\det F^a H^n_{\mathrm{dR}}(X)$.

\begin{bc}[\cite{Be}]
\textbf{(i)} $\rbe_X$ is an isomorphism;

\noindent \textbf{(ii)} $\mathrm{rk}(K_b^{(a)}(\fx))=\ord_{s=n+1-a}L(H^n(X),s)\overset{\text{$n$ odd}}{=}\sum_{n+1-a\leq p\leq \frac{n-1}{2}}h^{p,n-p}$, where the Hodge numbers $h^{p,q}:=\dim(H^{p,q}(X_{\CC}^{\text{an}}))$; and

\noindent \textbf{(iii)} the leading Taylor coefficient $L^*(H^n(X),n+1-a)\underset{\QQ^{\times}}{\sim}\det(\rbe_X)$, where $\det(\rbe_X)$ is computed relative to the natural $\QQ$-structures.
\end{bc}

\begin{rem}\label{r4a}
To translate (iii) into something more practical, write:
\begin{itemize}[leftmargin=0.5cm]
\item $\ell$ for the number in (ii) and $\{Z_1,\ldots,Z_{\ell}\}$ for a basis of $K_b^{(a)}(\fx)$;
\item $\{\omega_1,\ldots,\omega_{\kappa}\}$ for a basis of $\QQ$-dR cohomology $F^aH^n_{\mathrm{dR}}(X)$; and
\item $\{\gamma_1,\ldots,\gamma_d\}$ for a basis of $H_n(X_{\CC}^{\text{an}},\QQ)^{({\mathsf{F}_{\infty}})_*=(-1)^{a-1}}$ (note $d=\kappa+\ell$).
\end{itemize}
Let $\mathfrak{R}=[\mathfrak{R}_{ij}]$ be the $d\times d$ matrix whose first $\kappa$ (resp.~last $\ell$) columns are given by $\int_{\gamma_i}\omega_j$ (resp.~$\int_{\gamma_i}r_X(Z_{j-\kappa})$).  Then BC1 says that
\begin{equation}\label{e4h}
\det(\rbe_X)=\det((2\pi\ay)^{-a+1}\mathfrak{R})=(2\pi\ay)^{-d(a-1)}\det(\mathfrak{R})
\end{equation}
is a rational multiple of the $\ell^{\text{th}}$ derivative of $L(H^n(X),s)$ at $s=a-b$.  Note that in practice, the negative powers of $2\pi\ay$ will cancel with positive powers of $2\pi\ay$ in the regulator periods $\int_{\gamma}r(Z)$.
\end{rem}

\begin{example}\label{ex4a}
Let $X$ be a CY 3-fold with $h^{2,1}=1$, and $L(s):=L(H^3(X),s)$.  We describe the first two cases of BC1:
\begin{itemize}[leftmargin=0.8cm]
\item [\textbf{(a)}] Let $\{Z_1,Z_2\}\subset K_4^{(4)}(\fx)$ and $\{\gamma_1,\gamma_2\}\subset H_3(\xan,\QQ)^{(\mathsf{F}_{\infty})_*=-1}$ be bases;\footnote{Of course, the rank conjecture (part (i) of BC1) is not known in a single case beyond $X=$ the spectrum of a number field.  So what can be checked in principle is that there \emph{exist} cycles $Z_i$ for which the prediction on $L$-derivatives holds.} then 
\begin{equation}\label{eK4}
L''(0)\sim (2\pi\ay)^{-6}\det\left(\int_{\gamma_i}r_X(Z_j)\right).
\end{equation}
\item [\textbf{(b)}] Let $\{Z\}\subset K_2^{(3)}(\fx)$, $\{\omega\}\subset F^3H^3_{\mathrm{dR}}(X)$, $\{\delta_1,\delta_2\}\subset H_3(\xan,\QQ)^{(\mathsf{F}_{\infty})_*=1}$ be bases, and $\mathfrak{R}$ the $2\times 2$ matrix with columns $\int_{\delta_i}\omega$, $\int_{\delta_i}r_X(Z)$.  Then 
\begin{equation}\label{eK2}
L'(1)\sim (2\pi\ay)^{-4}\det(\mathfrak{R}).
\end{equation}
%note that modifying r_X by multiple of omega does not change the determinant, otherwise this wouldn't be well-defined.
\end{itemize}
\end{example}

\begin{rem}\label{r4b}
For $z\in \QQ$, we can replace $X$ in BC1 and Remark \ref{r4a} everywhere by the motive $M_z$ underlying the HS $\M_z\subseteq H^n(X_z)$ (and its Hodge numbers, etc.).  The upshot is that BC1 yields predictions for the arithmetic behavior of our higher normal functions, though in terms of the real or imaginary part of the section $\nu$ of $\J(\M(a))$ paired with topological cycles (rather than the pairing $V$ with $\mu$ in \eqref{e3h}).
\end{rem}

\subsection{Tempered Laurent polynomials}\label{Stemp}

There is a simple construction of 1-parameter families of CY $n$-folds carrying a family of $K_{n+1}^{(n+1)}$-classes.  Assume given a Laurent polynomial $\phi\in \QQ[x_1^{\pm 1},\ldots,x_{n+1}^{\pm 1}]$ with reflexive Newton polytope $\Delta$, and compactify the solutions of $1-t\phi(\ux)=0$ in $\GG_m^{n+1}\times (\PP^1\setminus\{0\})$ to a smooth family $\widetilde{\PP}_{\Delta}\times \PP^1 \supset \X\to \PP^1$ with CY fibers $X_t$, smooth for $t\in U$. Typically $\widetilde{\PP}_{\Delta}$ is a blowup of the toric variety along the base locus of the pencil.

\begin{defn}\label{d4a}
$\phi$ is \emph{tempered} if the toric coordinate symbol 
\begin{equation}\label{e4i'}
\{x_1,\ldots,x_{n+1}\}\in H_{\text{Mot}}^{n+1}(\GG_m^{n+1},\QQ(n+1))
\end{equation}
lifts to a class $\xi\in H^{n+1}_{\text{Mot}}(\X\setminus X_0,\QQ(n+1))$.\footnote{This is a condition on the face polynomials of $\phi$: essentially, the coordinate symbols on torus orbits define trivial cycles on the strata of the base locus.  It holds for many Laurent polynomials inducing LG-models, for example if $\phi$ is a Minkowski polynomial in 2 or 3 variables.}
\end{defn}

The construction of this family guarantees that $\F^n_e\M_e\cong \co(1)$, i.e.~$h=1$.  Define a section of this (i.e.~family of holomorphic forms) by
\begin{equation}\label{e4i}
\mu_t:=(2\pi\ay)^{-n}\mathrm{Res}_{X_t}\left(\frac{\mathrm{dlog}(\ux)}{1-t\phi}\right)	,
\end{equation}
and assume that there is a sub-VHS $\M\subset \H_{\cx_U/U}^n$ with Hodge numbers all $1$, satisfying the other hypotheses in \S\ref{S2}.  (Note that $\omega_t=(2\pi\ay)^n\mu_t$ is the class in $F^nH^n_{\mathrm{dR}}(X_t)$ for $t\in \QQ$.)  It is convenient to allow $z=Ct$ for some $C\in \QQ$, not necessarily $1$: in the examples below, $z$ will be the hypergeometric variable, and $t$ the ``geometric'' variable in this construction. 

Finally, set $V_0:=\langle \tilde{\nu}_{\xi},\mu\rangle$ as in \eqref{e3h}.  Two results of \cite[\S4]{DK} are then that\footnote{Strictly speaking, \eqref{e4j} makes use of our assumption that $L=L^{\dagger}$; see the proof of \cite[Thm.~9.7]{Ke}. In fact, \emph{with our assumptions on $L$}, by Theorem \ref{t3a} we see that \eqref{e4j} follows from \eqref{e4k} and $c_0=\mathsf{Q}_0^{-1}\in \QQ$.}
\begin{equation}\label{e4j}
LV_0=c_0
\end{equation}
and 
\begin{equation}\label{e4k}
\nabla_D[\tilde{\nu}_{\xi}]=(2\pi\ay)^n[\mu].
\end{equation}
In other words, $V$ resp.~$\nu$ are killed by $DL$ and $LD$ respectively. Furthermore, the periods over the vanishing cycle $\ve_0$ at $t=0$ are worked out in [op.~cit.], giving
\begin{equation}\label{e4l}
\int_{\ve_0}\mu = 1+\sum_{k\geq 1}[\phi^k]_{\uo}t^k
\;\;\;\text{and}\;\;\;
\frac{1}{(2\pi\ay)^{n}}\int_{\ve_0}\tilde{\nu}_{\xi}\underset{\QQ(1)}{\equiv}\log(t)+\sum_{k\geq 1}\frac{[\phi^k]_{\uo}}{k}t^k
\end{equation}
where $[\phi^k]_{\uo}$ are the constant terms in powers of $\phi$.

%%%%%%%%%%%%
%%%%%%%%%%%%

\section{The 14 hypergeometric cases}\label{S5}

Given a sequence $(\gamma_1,\ldots,\gamma_{\ell})$ of nondecreasing integers summing to $0$, we can write $$\frac{\prod_{\gamma_i<0}(T^{-\gamma_i}-1)}{\prod_{\gamma_i>0}(T^{\gamma_i}-1)}=\frac{\prod_{j=1}^m(T-e^{2\pi\ay \fa_j})}{\prod_{j=1}^m(T-e^{2\pi\ay\fb_j})},\;\;\;\fa_j,\fb_j\in (0,1]\cap \QQ.$$
One checks that $$a_k:=\frac{\prod_j [\fa_j]_k}{\prod_j[\fb_j]_k}=C^{-k}\frac{\prod_{\gamma_i<0}(-\gamma_i k)!}{\prod_{\gamma_i>0}(\gamma_i k)!},$$
where $C=\prod_i \gamma_i^{\gamma_i}$.\footnote{This $C$ is $\lambda^{-1}$ of \S\ref{S2cl}.}  Writing $z=Ct$, the hypergeometric function associated to the data $(\fa,\fb)$ (or $\gamma$), and killed by the operator $L=\prod_j (D+\fb_j-1)- z\prod_j(D+\fa_j)$, is given by $\sum_{k\geq 0}a_k z^k=\sum_{k\geq 0} a_k C^k t^k$.  A formula for the Hodge numbers of the VHS over $\PP^1_z\setminus \{0,1,\infty\}$ with underlying $\mathcal{D}$-module $\mathcal{D}/\mathcal{D}L$ is given in \cite{Fe} (also see \cite{RRV}).  In this paper we are mainly interested in the case where the $\fb_j=1$, which forces maximal unipotent monodromy at $0$ and the Hodge numbers to all be $1$.

Doran and Morgan \cite{DM} classified the 14 real types of $\QQ$-VHS $\M$ over $U=\PP^1\setminus \{0,1,\infty\}$ satisfying the hypotheses of \S\ref{S2} with $d=1$ and $n=3$.  They are all hypergeometric.  Thus in each case we have Hodge numbers $(1,1,1,1)$, maximal unipotent monodromy at $0$, conifold monodromy at $z=1$ ($t=C^{-1}$), and monodromy at $\infty$ determined by the hypergeometric indices $\fa$. If we drop the reflexivity/CY requirement, then all arise from the tempered Laurent polynomial construction above, and \eqref{e4i'}-\eqref{e4l} still go through.  (In five cases,\footnote{See the table for four of these.  In general, one should have in mind that $[\phi^k]_{\uo}=\frac{C^k}{k!^4}\prod_{j=1}^4 [\fa_j]_k$ is a product of multinomial symbols.} there is a reflexive Laurent polynomial for the family, and in five more, for some base change of it.)  This puts us in the situation of Example \ref{ex3a}, with one normal function of $K_4$-type singular at $0$. In particular, \eqref{e4l} becomes for $z\in (0,1)$
\begin{equation}\label{e50}
\textstyle\int_{\ve_0}\mu=\sum_{k\geq 0}\frac{\prod_{j=1}^4[\fa_j]_k}{k!^4}z^k
\;\;\;\text{and}\;\;
\frac{\int_{\ve_0}r(\xi)}{(2\pi\ay)^3} = \log(\frac{z}{C})+\sum_{k\geq 1}\frac{\prod_{j=1}^4[\fa_j]_k}{k!^4\cdot k}z^k 
\end{equation}

To check BC1 for $M_z$ at (some) $z\in \QQ$, we need either a $K_2$-cycle or an additional $K_4$-cycle --- or at least, a normal function which BHC predicts to come from such a cycle.  Consider the base change by $\rho\colon \PP^1_{\hat{z}}\to \PP^1_z$ sending $\hat{z}\mapsto \hat{z}^2=z$, so that $\rho^*\M=:\widehat{\M}$ has degree $d=2$.  As $\widehat{\M}$ has two conifold points and at most one invariant class at $\infty$, by \eqref{e3e} the rank of the pure Hodge structure $\IH^1(\PP^1,\widehat{\M})$ is $0$ or $1$.  Thus $F^4\IH^1(\PP^1,\widehat{\M})=\{0\}$ and $h=1$, putting us in the situation of Example \ref{ex3b}, with a second normal function $\nu_{\infty}$ which is nonsingular away from $\infty$.  Evidently the resulting multivalued function $V_{\infty}$ satisfies
\begin{equation}\label{e5a}
LV_{\infty}=c_{\infty}\hat{z}=c_{\infty}\sqrt{z}.
\end{equation}

Assuming the BHC, the following table tells us which kind of cycle $\mathfrak{Z}$ to expect for $\nu_{\infty}(=\nu_{\mathfrak{Z}})$ in each of the 14 hypergeometric cases:

\[\begin{tikzpicture}[scale=0.97]
\draw [gray] (-2.8,-0.5) -- (10,-0.5);
\draw [gray] (0,0) -- (0,-9.4);
\draw [gray] (2.4,0) -- (2.4,-9.4);
\draw [gray] (5,0) -- (5,-9.4);
\draw [gray] (7.6,0) -- (7.6,-9.4);
\draw [gray] (10,0) -- (10,-9.4);
\draw [gray] (-2.8,-1.3) -- (10,-1.3);
\draw [gray] (-2.8,-3.7) -- (10,-3.7);
\draw [gray] (-2.8,-6) -- (10,-6);
\draw [gray] (-2.8,-7) -- (10,-7);
\draw [gray] (-2.8,-7.8) -- (10,-7.8);
\draw [gray] (-2.8,-8.6) -- (10,-8.6);
\draw [gray] (-2.8,-9.4) -- (10,-9.4);
\node [] at (-1.5,-0.2) {\large Type};
\node [] at (1.2,-0.2) {\large Ia};
\node [] at (3.7,-0.2) {\large Ib};
\node [] at (6.2,-0.2) {\large II};
\node [] at (8.7,-0.2) {\large IV};
\node [] at (-1.5,-1) {\# of cases};
\node [] at (1.2,-0.9) {$7$};
\node [] at (3.7,-0.9) {$3$};
\node [] at (6.2,-0.9) {$3$};
\node [] at (8.7,-0.9) {$1$};
\node [] at (-1.5,-1.8) {example $\fa$};
\node [] at (-0.67,-2.5) {$C$};
\node [] at (1.2,-2.5) {$5^5$};
\node [] at (3.7,-2.5) {$2^{10}$};
\node [] at (6.2,-2.5) {$3^6$};
\node [] at (8.7,-2.5) {$2^8$};
\node [] at (-0.7,-3.2) {$\phi$};
\node [] at (1.2,-1.7) {$(\tfrac{1}{5},\tfrac{2}{5},\tfrac{3}{5},\tfrac{4}{5})$};
\node [] at (3.7,-1.7) {$(\tfrac{1}{4},\tfrac{1}{2},\tfrac{1}{2},\tfrac{3}{4})$};
\node [] at (6.2,-1.7) {$(\tfrac{1}{3},\tfrac{1}{3},\tfrac{2}{3},\tfrac{2}{3})$};
\node [] at (8.7,-1.7) {$(\tfrac{1}{2},\tfrac{1}{2},\tfrac{1}{2},\tfrac{1}{2})$};
\node [] at (1.2,-3.2) {$\tfrac{(1+\Sigma_{i=1}^4 x_i)^5}{x_1x_2x_3x_4}$};
\node [] at (3.7,-3.2) {\tiny $\tfrac{(1-\Sigma_{i=1}^3 x_i)^4(1-x_4)^2}{x_1x_2x_3x_4}$};
\node [] at (6.3,-3.2) {\tiny $\substack{2 \\ \prod \\ {i=1}}\mspace{-5mu}\tfrac{(\mspace{-2mu}1\mspace{-3mu}+x_{{}_{1+2i}}\mspace{-5mu}+x_{{}_{2+2i}}\mspace{-2mu})^3}{x_{{}_{1+2i}}x_{{}_{2+2i}}}$};
\node [] at (8.7,-3.2) {$\substack{4\\ \prod \\ {i=1}} \tfrac{(1-x_i)^2}{x_i}$};
\node [] at (-1.4,-4.9) {$\psi_{\infty}\M$ type};
\foreach \i in {0,...,3} {\draw [thick,gray] (0.4+2.5*\i,-4) -- (0.4+2.5*\i,-5.7) -- (2.1+2.5*\i,-5.7);}
\filldraw [blue] (0.4,-4.2) circle (2pt);
\filldraw [blue] (0.9,-4.7) circle (2pt);
\filldraw [blue] (1.4,-5.2) circle (2pt);
\filldraw [blue] (1.9,-5.7) circle (2pt);
\filldraw [blue] (2.9,-4.2) circle (2pt);
\filldraw [blue] (3.4,-5.2) circle (2pt);
\filldraw [blue] (3.9,-4.7) circle (2pt);
\filldraw [blue] (4.4,-5.7) circle (2pt);
\filldraw [blue] (5.4,-4.7) circle (2pt);
\filldraw [blue] (5.9,-4.2) circle (2pt);
\filldraw [blue] (6.4,-5.7) circle (2pt);
\filldraw [blue] (6.9,-5.2) circle (2pt);
\filldraw [blue] (7.9,-5.7) circle (2pt);
\filldraw [blue] (8.4,-5.2) circle (2pt);
\filldraw [blue] (8.9,-4.7) circle (2pt);
\filldraw [blue] (9.4,-4.2) circle (2pt);
\draw [red,<-] (3.5,-5.1) -- (3.8,-4.8);
\node [red] at (3.8,-5.1) {\small $N$};
\draw [red,<-] (5.5,-4.6) -- (5.8,-4.3);
\draw [red,<-] (6.5,-5.6) -- (6.8,-5.3);
\draw [red,<-] (8,-5.6) -- (8.3,-5.3);
\draw [red,<-] (8.5,-5.1) -- (8.8,-4.8);
\draw [red,<-] (9,-4.6) -- (9.3,-4.3);
\node [] at (-1.5,-6.6) {$(\psi_{\infty}\M)_{T_{\infty}^2}\mspace{-2mu}(-1)$};
\node [] at (1.2,-6.6) {$0$};
\node [] at (3.7,-6.6) {$\QQ(-3)$};
\node [] at (6.2,-6.6) {$0$};
\node [] at (8.7,-6.6) {$\QQ(-4)$};
\node [] at (-1.5,-7.4) {$a$};
\node [] at (1.2,-7.4) {$2$};
\node [] at (3.7,-7.4) {$3$};
\node [] at (6.2,-7.4) {$2$};
\node [] at (8.7,-7.4) {$4$};
\node [] at (-1.5,-8.2) {$\mathfrak{Z}/V_{\infty}$ type};
\node [] at (1.2,-8.2) {$K_0$};
\node [] at (3.7,-8.2) {$K_2$};
\node [] at (6.2,-8.2) {$K_0$};
\node [] at (8.7,-8.2) {$K_4$};
\node [] at (-1.5,-9) {what is $\fz$?};
\node [] at (1.2,-9) {\SMALL\cite{MW,KW}};
\node [] at (3.7,-9) {see \S\ref{S6}};
\node [] at (6.2,-9) {\SMALL\cite{Wa2}};
\node [] at (8.7,-9) {see \S\ref{S8}};
\end{tikzpicture}\]

\noindent By the uniqueness result Prop.~\ref{p2b} (and Example \ref{ex2b}), we have a \emph{formula} for $V_{\infty}$ without constructing the cycle!  Namely, there exists a lift $\tilde{\nu}_{\infty}$ on a neighborhood of $0$ with $V_{\infty}(0)=0$, and then in all 14 cases
\begin{equation}\label{e5b}
V_{\infty}(z)=c_{\infty}\W_{2}(z)=c_{\infty}\sqrt{z}\sum_{k\geq 0}\frac{\prod_{j=1}^4[\fa_j+\frac{1}{2}]_k}{[\frac{1}{2}]_{k+1}^4}z^k.
\end{equation}
For the first example of the table, \eqref{e5b} is essentially the function that showed up in \cite{Wa} as the generating function for open GW invariants on the Fermat quintic.

By Theorem \ref{t3a}, we know (assuming $\nu_{\infty}$ is motivic) that $c_{\infty}\in \qb$.  Passing to the algebro-geometric parameter $t=C^{-1}z$, we can also use the Theorem (with $p=a$) to produce an analogue of \eqref{e4k} for $V_{\infty}$.  Write $\hat{t}$ for $\sqrt{t}$, i.e.~the coordinate on the 2:1 base-change, and $\hat{D}=D_{\hat{t}}$.  Notice that \eqref{e5b} fixes $V_{\infty}$ but not $\tilde{\nu}_{\infty}$, which still has an ambiguity in $\mathcal{F}^a\mathcal{M}$ that Theorem \ref{t3a} tells us how to fix, yielding
\begin{equation}\label{e5c'}
D\tilde{\nu}_{\infty}=	\mathsf{Q}_0(2\pi\ay)^3c_{\infty}C^{\frac{1}{2}}\sqrt{t}\mu=:c'(2\pi\ay)^3\hat{t}\mu\;\implies\;\hat{D}\tilde{\nu}_{\infty}=2c'\hat{t}\omega,
\end{equation}
where $\omega=(2\pi\ay)^3\mu$.  This has two consequences.  First, for the regulator periods, we get
\begin{equation}\label{e5d'}
\int_{\ve_j}\tilde{\nu}_{\infty}=c'\int\frac{dt}{t}\int_{\ve_j}\sqrt{t}\omega_t=c'\int\frac{dt}{\sqrt{t}}\int_{\ve_j}\omega_t.
\end{equation}
Second, by \eqref{e5c'} we get a homogeneous equation for $\tilde{\nu}_{\infty}$:  it is killed by the operator $L(\hat{D}+1)\frac{1}{\hat{t}}=L\frac{1}{\hat{t}}\hat{D}$.  

We can use \eqref{e5d'} together with the Mellin-Barnes formula to compute the $K_2$ regulator periods (case $a=3$), which will be done in \S\ref{S8}.  (For the $K_4$ periods worked out next in \S\ref{S6}, we use a symmetry of the family instead.)  Notice that the choice of lift, while useful for computation, does not affect the determinant in Example \ref{ex4a}(b), since any two lifts differ by a multiple of the holomorphic form class.

%%%%%%%%%%%%%%%
%%%%%%%%%%%%%%%

\section{Rank two regulators:  Beilinson for $K_4$}\label{S6}

For the purpose of checking BC1, we need to compute regulator periods $\int_{\gamma} r(\mathfrak{Z})=\int_{\gamma}\pi(\nu_{\mathfrak{Z}})$ rather than $V_{\infty}$.  The idea is to first determine the period over the vanishing cycle $\ve_0$, then use Frobenius deformation to deduce those over $\ve_1,\ve_2,\ve_3$ as needed (also for $\mu$).  Note that if $z\in (0,1)$, $\ve_0$ and $\ve_2$ are the $\mathsf{F}_{\infty}$-anti-invariant cycles (for $K_4$); whereas $\ve_1$ and $\ve_3$ are the invariant ones (for $K_2$).  This is immediate from \eqref{e2b} and Example \ref{ex2a}, since the periods $(2\pi\ay)^{3+m}\e_m(z)$ of the $\QQ$-dR form $\omega=(2\pi\ay)^3\mu$ are real (marking $\ve_m$ as $\mathsf{F}_{\infty}$-invariant) for $m$ odd. In this section we work out the Frobenius deformation approach for $K_4$. 

Now from the table we expect a second $K_4$-class in one case, namely $\fa=(\tfrac{1}{2},\tfrac{1}{2},\tfrac{1}{2},\tfrac{1}{2})$.  Here we have $\M\subset \H^3$ of the CY 3-fold family\footnote{Granted, the fibers are singular, but these are quotient singularities which are irrelevant for the Hodge- and cycle-theoretic constructions we consider.} $X_t=\{1-t\phi=0\}\subset (\PP^1)^4$ with $\phi=\prod_{i=1}^4(1-x_i)^2/x_i$, and $\widehat{\M}\subset \H^3$ of the family $\widehat{X}_{\hat{t}}=\{1-\hat{t}\hat{\phi}=0\}$ with $\hat{\phi}=\prod_{i=1}^4 (\hat{x}_i-\hat{x}_i^{-1})$.  The point is simply that $[\phi^k]_{\uo}=\binom{2k}{k}^4=[\hat{\phi}^{2k}]_{\uo}$; geometrically, sending $(\hat{t},\hat{x}_i)\mapsto (\hat{t}^2,\hat{x}_i^2)=(t,x_i)$ induces 8:1-orbifolding maps $\widehat{X}_{\hat{t}}\to X_{\hat{t}^2}$ under which $\M$ (resp.~$\mu,\xi$) pulls back to $\widehat{\M}$ (resp.~$8\hat{\mu},16\hat{\xi}$).

For the holomorphic period downstairs we have
\begin{equation}\label{e5c}
\textstyle \e_0=\int_{\ve_0}\mu=\sum_{k\geq 0}\binom{2k}{k}^4 t^k=\sum_{k\geq 0}\frac{[\frac{1}{2}]_k^4}{k!^4}z^k,\;\;\;z=256t.
\end{equation}
Writing $\hz=16\hat{t}$, there is a birational involution \cite[\S4.2]{BKSZ}
\begin{equation}\label{e5d}
\mathcal{I}\colon(\hz;\;\hx_1,\hx_2,\hx_3,\hx_4)\mapsto \left(\frac{1}{\hz};\;\frac{1+\hx_2}{1-\hx_2},\frac{\hx_1-1}{\hx_1+1},\frac{1+\hx_4}{1-\hx_4},\frac{\hx_3-1}{\hx_3+1}\right)
\end{equation}
of the total space $\widehat{\X}$, under which we pull back the coordinate symbol $\hat{\xi}\in \mathrm{CH}^4(\widehat{\X}\setminus \hat{X}_0,4)$ to $\mathfrak{Z}:=\mathcal{I}^*\hat{\xi}\in \mathrm{CH}^4(\widehat{\X}\setminus \hat{X}_{\infty},4)$.  Since 
\begin{equation}\label{e5e}
\nabla_{\hat{D}} [\nu_{\mathfrak{Z}}]=-\mathcal{I}^*\nabla_{\hat{D}}[\nu_{\hat{\xi}}]=-(2\pi\ay)^3\mathcal{I}^*[\hat{\mu}]=(2\pi\ay)^3\hat{z}[\hat{\mu}],
\end{equation}
this yields a normal function with the expected class.  Note that this cycle only lives ``upstairs''; pushing it forward to $\X$ gives zero because the cycles from $\widehat{X}_{\hat{t}}$ and $\widehat{X}_{-\hat{t}}$ cancel.

The construction of the second cycle $\fz$ is used below merely to predict for which values of $t$ we expect BC1 to hold.  For $\hat{t}=\frac{a}{b}\in \QQ$, the natural model $\widehat{\fx}_{\hat{t}}\to \text{Spec}(\ZZ)$ is given by $b\prod_i x_i=a\prod_i (x_i-x_i^{-1})$.  If $a=1$, i.e.~$\hat{t}^{-1}\in \ZZ$, then there is no prime over which the fiber is the ``toric boundary'' of $(\PP^1)^4$, and so $\hat{\xi}$ extends across the model.  Since $\mathfrak{Z}$ is obtained by pulling this back under $\hat{t}\mapsto (256\hat{t})^{-1}$, it will extend to $\widehat{\fx}_{\hat{t}}$ for $256\hat{t}\in \ZZ$.  Since we need both cycles to extend to the integral model, this leaves $\hat{t}=2^{-\ell}$ (resp. $t=4^{-\ell}$) for $0\leq \ell\leq 8$; and we have to omit the conifold fiber, which corresponds to $\ell=4$.  Moreover, the mirror images under $\mathcal{I}$ are the same case.  So in the original coordinate, one expects these cycles to confirm BC1 at $t_k=4^{-4-k}$ (resp.~$z_k=4^{-k}$) for $k=1,2,3,4$.

\begin{rem}\label{rZ}
In general, for cycles with singularity at (say) $\hat{t}=\infty$, informally speaking one must avoid having the fiber over $\infty$ reappear in the integral model for the cycle to extend.  So for the cycle to extend to the ``obvious'' model for $1-\hat{t}_0\hat{\phi}(\hat{\ux})=0$, one might expect to require $\hat{t}_0\in \ZZ$, though as in the present example we can sometimes do a bit better ($256\hat{t}_0\in \ZZ$).  However, this heuristic does suggest that for the type IIb/$K_2$ cases (where the \emph{only} normal function has singularity at $\infty$), we may be better off working about $\infty$ --- which is exactly what we will do in \S\ref{S8}.
\end{rem}

Using \eqref{e4k} and \eqref{e5e} (in the form $\nabla_D [\nu_{\mathfrak{Z}}]=\sqrt{t}[\mu]$) to integrate \eqref{e5c} yields 
\begin{align}\label{e6a}
\frac{\int_{\ve_0}r(\xi_t)}{(2\pi\ay)^3}&=\log(t)+\sum_{k>0}\frac{\binom{2k}{k}^4}{k}t^k\;\;\;\;\;\;\;\text{and}
\\
\frac{\int_{\ve_0}r(\mathfrak{Z}_t)}{(2\pi\ay)^3}&=\sqrt{t}\sum_{k\geq 0}\frac{\binom{2k}{k}^4}{k+\frac{1}{2}}t^k.\label{e6a'}
\end{align}
for $t\in (0,4^{-4})$.  Indeed, $\nu_{\mathfrak{Z}}$ is zero at the origin, so (for $0\leq j\leq 3$) we have $\int_{\ve_j}\tfrac{\tilde{\nu}_{\mathfrak{Z}_t}}{(2\pi\ay)^3}=\int_0\tfrac{dt}{t}\int_{\ve_j}\sqrt{t}\mu_t=\int_0\tfrac{dt}{\sqrt{t}}\int_{\ve_j}\mu_t$.  Since $\int_0\tfrac{dt}{\sqrt{t}}$ commutes with Frobenius deformation, we get exactly the same thing by directly applying the latter to $\int_{\ve_0}\tfrac{\tilde{\nu}_{\mathfrak{Z}_t}}{(2\pi\ay)^3}$.  That is, we replace $k$ in RHS\eqref{e6a} by $k+s$, differentiate twice in $s$ (dividing by $2\pi\ay$ each time in view of \eqref{e2b}), and set $s=0$ to produce
\begin{align}\label{e6b}
&\frac{\int_{\ve_2}r(\mathfrak{Z}_t)}{(2\pi\ay)^3}=\frac{\sqrt{t}}{-8\pi^2}\left(\sum_{k\geq 0}\frac{\binom{2k+2s}{k+s}^4}{k+s+\frac{1}{2}}t^{k+s}\right)''_{s=0}=
\\
\tfrac{\sqrt{t}}{-4\pi^2}\sum_{k\geq 0}\textstyle\binom{2k}{k}^4t^k&\left\{\tfrac{4G_k'(k+\frac{1}{2})^2-8G_k(k+\frac{1}{2})+1}{(k+\frac{1}{2})^3}+\tfrac{8G_k(k+\frac{1}{2})-1}{(k+\frac{1}{2})^2}\log (t)+\tfrac{1}{k+\frac{1}{2}}\log^2(t)\right\}\notag
\end{align}
where $H_k=\sum_{\ell=1}^k\tfrac{1}{\ell}$, $H_k'=\sum_{\ell=1}^k\tfrac{1}{\ell^2}$, $G_k:=H_{2k}-H_k$, and $G_k':=8G_k^2-2H_{2k}'+H_k'+\zeta(2)$.

Obtaining the second regulator period for $\xi_t$ is a bit more subtle since $\nu_{\xi_t}$ is singular at $0$ and $\int_0 \tfrac{dt}{t}$ makes no sense. The trick is to view $\int_{\ve_0}\tilde{\nu}_{\xi_t}$ as itself the first Frobenius deformation of something --- which is formally a constant function!  Namely, defining 
\begin{equation}\label{e6c}	(2\pi\ay)^3\sum_{k\geq 0}\frac{\binom{2k+2s}{k+s}^4}{k+s}t^{k+s}=: \sum_{m\geq -1}(2\pi\ay)^m \rbe_m(t)s^m=:\mathfrak{R}(t,s)
\end{equation}
(with $\rbe_{-1}=(2\pi\ay)^4$) we observe that 
$$
(2\pi\ay)^{-3}D\mathfrak{R}=\sum_{k\geq 0}\binom{2k+2s}{k+s}^4t^{k+s}=\bE(z,s)
$$
and $T_0\mathfrak{R}=e^{2\pi\ay s}\mathfrak{R}$.  These two equations uniquely determine $\mathfrak{R}$ and are also satisfied (truncating at $s^4$) by the generating function of regulator periods for $\xi_t$.  So we conclude that $\rbe_m(t)=\int_{\ve_m}\tilde{\nu}_{\xi_t}$ for $0\leq m\leq 3$, which recovers LHS\eqref{e6a} and also yields
\begin{align}\label{e6d}
\frac{\int_{\ve_2}r(\xi_t)}{(2\pi\ay)^3}=\frac{-1}{4\pi^2}\left(\sum_{k\geq 0}\frac{\binom{2k+2s}{k+s}^4}{k+s}t^{k+s}\right)_{\text{coeff.~of $s^2$}}=
\\
\frac{-1}{4\pi^2}\left(\begin{matrix}\left\{-8\zeta(3)+\sum_{k>0}\binom{2k}{k}^4 \tfrac{4G_k'k^2-8G_kk+1}{k^3}t^k \right\}+\\ \left\{4\zeta(2)+\sum_{k>0}\binom{2k}{k}^4\tfrac{8G_kk-1}{k^2}t^k\right\}\log(t)+\\ \left\{\sum_{k>0}\binom{2k}{k}^4\frac{1}{k}t^k\right\}\log^2(t)+\frac{1}{6}\log^3(t)\end{matrix}\right).\notag
\end{align}

We summarize these results in the

\begin{thm}\label{tK4}
For $z\in (0,1)$ \textup{(}$t\in (0,4^{-4})$\textup{)}, the real regulator determinant associated to the motive $M_z$, normalized as in RHS\eqref{eK4}, is given by
\begin{equation}\label{eRK4}
r(t):= \eqref{e6a'}\times\eqref{e6d}-\eqref{e6a}\times\eqref{e6c}.
\end{equation}
The corresponding pair of cycles extend to an integral model at $t_k=4^{-4-k}$ \textup{(}$k=1,2,3,4$\textup{)}.
\end{thm}

Now the representation \eqref{eRK4} is valid at the $t_k$; and by Example \ref{ex4a}(a), $r(t_k)$ should be a rational multiple of $L''({M}_{z_k},0)$.  According to our numerical implementation in MAGMA, we get
\begin{equation}\label{eNUM4}
\frac{L''(0)}{r(t)} \;\overset{\text{\color{red}num}}{=} \;4,\frac{64}{3},8,4\;\;\; \text{ for }\;\;\;t=4^{-5},4^{-6},4^{-7},4^{-8},
\end{equation}
where 
$L(s)$ means $L(M_{2^8t},s)$.

\begin{rem}\label{rVoight}
This numerical evidence is most interesting if one knows that these points are not special, i.e.~that the Zariski closure of the Galois representation is the full $\mathrm{GSp}_4$.  We are grateful to J.~Voight for running a numerical computation of the Sato-Tate moments in Magma for $M_{z_4}$, which shows that this is indeed the case.
\end{rem}

%%%%%%%%%%%%%%%
%%%%%%%%%%%%%%%

\section{Intermezzo on Hadamard products}\label{S7}

In this section, we first briefly recall the nature of Hadamard products from the point of view of periods and hypergeometric VHS.  Then we show how to use them to compute the regulator period $\int_{\varepsilon_0}R_{Z_t}$, for a $K_4$-cycle $Z_t$ equivalent to the coordinate symbol $\xi_t$ from \S\ref{S6}.  In \S\ref{S7}, the same technique will be used in a more involved computation for $K_2$ of a different hypergeometric CY3 family.

\subsection{Periods}

Let $W^{(j)}\to \PP^1_s$, $j=1,2$, be two families of (generically) smooth projective varieties, with fibers of dimension $d_j$.  Write $\mu^{(j)}$ for families of holomorphic $d_j$-forms (on these fibers), and $\varepsilon^{(j)}$ for families of $d_j$-cycles well-defined on a punctured disk about $s=0$, with periods of the form $\pi^{(j)}(s):=\int_{\varepsilon^{(j)}_s}\mu^{(j)}_s=\sum_{k\geq 0}a^{(j)}_k s^k$.  

Form the family of Hadamard products $V_t:=W^{(1)}\times_{\PP^1_s}\rho_t^* W^{(2)}$, where $\rho_t\in \mathrm{Aut}(\PP^1)$ sends $s\mapsto t/s$.  These $(d_1+d_2+1)$-folds satisfy a sort of analytic ``Freshman's Dream'' for periods.  Namely, let $\pi_{(j)}$ be the obvious projections, and write $\mu_t=\pi_{(1)}^*\omega^{(1)}\wedge \pi_{(2)}^*\rho_t^* \omega^{(2)}\wedge \tfrac{ds}{2\pi\ay s}$ for holomorphic forms,\footnote{The assumption is that $\omega^{(j)}$ are holomorphic sections of the extended Hodge bundle, vanishing at $s=\infty$.  So the vanishing of $\omega^{(1)}$ resp.~$\rho_t^*\omega^{(2)}$ at $s=\infty$ resp.~$0$ cancels out the poles of $\tfrac{ds}{s}$ at $0$ and $\infty$.} and $\varepsilon_t=\cup_{|s|=\epsilon}\varepsilon_s^{(1)}\times \varepsilon_{t/s}^{(2)}$ for topological cycles (with $0<|t|\ll \epsilon\ll 1$).  Then we have
\begin{equation}\label{eHP}
\pi(t):=\int_{\varepsilon_t}\omega_t=\frac{1}{2\pi\ay}\oint_{|s|=\epsilon}\sum_{k\geq 0}a^{(1)}_ks^k\sum_{\ell\geq 0}a^{(2)}_{\ell}\frac{t^{\ell}}{s^{\ell}}\frac{ds}{s}=\sum_{k\geq 0}a^{(1)}_ka^{(2)}_k t^k.
\end{equation}

In the case where the periods are hypergeometric, with $\pi^{(j)}(s)=\sum_{k\geq 0}\prod_{i=1}^{m_{j}}\tfrac{[\fa^{(j)}_i]_k}{[\fb^{(j)}_i]_k}C_{(j)}^k s^k$, \eqref{eHP} takes the form $\sum_{k\geq 0}\prod_{i=1}^{m_{1}+m_{2}}\tfrac{[\fa_i]_k}{[\fb_i]_k}C^k t^k$, where $C=C_{(1)}C_{(2)}$ and $(\fa,\fb)=(\fa_{(1)}\cup \fa_{(2)},\fb_{(1)}\cup \fb_{(2)})$ is given by concatenating hypergeometric indices. 

\subsection{Algebraic cycles}
We are going to construct a $K_4$-cycle on the hypergeometric CY3-family of type $(\fa,\fb)=((\tfrac{1}{2},\tfrac{1}{2},\tfrac{1}{2},\tfrac{1}{2}),(1,1,1,1))$ via $K_2 \cup K_2$ for two copies of ``Legendre'', i.e.~$(\fa,\fb)=((\tfrac{1}{2},\tfrac{1}{2}),(1,1))$.  Here the elliptic surfaces $\mathcal{E}^{(j)}\to \PP^1_s$ are both given by the Kodaira compactification of $1-s\phi(x_j,y_j)=0$ with $\phi(x,y)=x^{-1}y^{-1}(x-1)^2(y-1)^2$.  (The fibers are $\mathrm{I}_4$, $\mathrm{I}_1$, $\mathrm{I}_1^*$, and this is precisely the family from \S\ref{Sreg}.)  As above, let $\rho_t$ denote the involution of $\PP^1$ (or $\mathbb{G}_m$) induced by $s\mapsto t/s$, and let $\tx_t\overset{\pi}{\to}\PP^1_s$ denote some desingularization of the fiber product $X_t=\mathcal{E}^{(1)}\times_{\PP^1}\rho_t^*\mathcal{E}^{(2)}\to \PP^1_s$ with fiber $E^{(1)}_s\times E^{(2)}_{t/s}$ over $s$.

For simplicity, we are going to construct the cycle that's defined without passing to the 2:1-basechange.  (Obviously the same idea will work for the second cycle $\fz$; this is similar enough to the $K_2\cup K_0$ example below that we leave it to the reader.)  We use two copies of the cycle from \S\ref{Sreg}.  Let $Z^{(j)}$ denote the extension of the symbol $\{x_j,y_j\}$ to an element of $\mathrm{CH}^2(\mathcal{E}^{(j)}\setminus E^{(j)}_0,2)$, and $\pi_1:\tilde{X}_t\to \mathcal{E}^{(1)}$, $\pi_2\colon \tilde{X}_t\to \rho_t^*\mathcal{E}^{(2)}$ the projections.  Then $\pi_1^*Z^{(1)}\cup \pi_2^*\rho^*_tZ^{(2)}\in \mathrm{CH}^4(\pi^{-1}(\GG_m),4)$ extends to $Z_t\in \mathrm{CH}^4(\tx_t,4)$ since $Z^{(1)}_{\infty}$ and $(\rho_t^*Z^{(2)})_0$ vanish identically \emph{as precycles} in $Z^2(E_{\infty}^{(1)},2)$ resp.~$Z^2((\rho_t^*\mathcal{E}^{(2)})_0,2)\cong Z^2(E^{(2)}_{\infty},2)$.

\begin{rem}
(i) In general, if $\tilde{X}_t$ is a desingularization of a Hadamard product of two elliptic surfaces $\mathcal{E}^{(j)}$, one can ask when the product of two precycles $Z^{(j)}$ (as above) over $\mathbb{G}_m$ can be extended to a \emph{closed} precycle on all of $\tilde{X}_t$.  If both are ``usual algebraic cycles'' (of $K_0$ type), Zariski closure will suffice; but if at least one is a ``higher'' precycle, the desingularization complicates matters.  A sufficient condition is to have each $Z^{(j)}_{\infty}$ vanish in motivic cohomology of $E_{\infty}^{(j)}$ --- this is what happens with $K_2\cup K_0$ in \S\ref{S8} --- which is weaker than the ``zero on the nose'' property we had here.  However, if they only vanish in the higher Chow group (motivic Borel-Moore homology) of $E_{\infty}^{(j)}$, then the extension can provably fail to exist.

(ii) The specific cycle being constructed here is equivalent to (the extension of) $\{z_1,z_2,z_3,z_4\}$ in $\mathrm{CH}^4(Y_t,4)$ with $Y_t$ the compactification of $t^{-1}\prod z_i=\prod(z_i-1)^2$ in $(\PP^1)^4$, since we just take the product of the equations $(s/t)x_2y_2=(x_2-1)^2(y_2-1)^2$ and $(1/s)x_1y_1=(x_1-1)^2(y_1-1)^2$.
\end{rem}

\subsection{Regulators} Moving to the regulator currents, write $(T_j,\Omega_j,R_j)$ for the triples associated to $\pi_1^*Z^{(1)}$ and $\pi_2^*\rho_t^*Z^{(2)}$ as precycles on $\tx_t$.  Here $\Omega_1=\tfrac{dx_1}{x_1}\wedge \tfrac{dy_1}{y_1}=\omega^{(1)}\wedge \tfrac{ds}{s}$, $\Omega_2= \tfrac{dx_2}{x_2}\wedge \tfrac{dy_2}{y_2}=\rho_t^*\omega^{(2)}\wedge (-\tfrac{ds}{s})$, and $T_j$ is the chain parametrized by $(x_j,y_j)\in \RR_-^{\times 2}$, which can be thought of as the family of nonvanishing cycles $\beta_s$ over a cut from the north pole to the conifold.  Since these cuts (basically $[0,\tfrac{1}{16}]$ and $[16t,\infty]$) don't meet, and $\tfrac{dt}{t}\wedge\tfrac{dt}{t}=0$, we get $(T,\Omega,R):=$
\begin{align}\label{e91}
(T_1,\Omega_1,R_1)\cup (T_2,\Omega_2,R_2) &= (T_1\cap T_2,\Omega_1\wedge\Omega_2,R_1\Omega_2+(2\pi\ay)^2\delta_{T_1}R_2)\\
&=(0,0,-R_1\wedge \rho_t^*\omega^{(2)}\wedge\tfrac{ds}{s}+(2\pi\ay)^2\delta_{T_1}R_2)\nonumber
\end{align}
which ensures that $R\in \mathcal{D}^3(\tx_t)$ is closed for general $t$. 

\begin{rem}
The first line of \eqref{e91} is cup-product in Deligne-Beilinson cohomology of $\tx_t$, viewed as a cone complex.  \emph{Warning:}  $T$ and $\Omega$ are not zero on the total space $\cup_t \tx_t$ of the family of CY3s; for instance, $\Omega$ is the obvious $(4,0)$ form $\bigwedge \tfrac{dz_i}{z_i}$, which pulls back to $0$ on each fiber.
\end{rem}

Taking $0<|t|\ll \epsilon\ll 1$, we form the 3-cycle\footnote{I give $|s|=\epsilon$ the clockwise orientation here.} $\varepsilon_{0,t}=\cup_{|s|=\epsilon}\alpha_s^{(1)}\times \alpha_{t/s}^{(2)}$ where $\alpha$ is the vanishing 1-cycle near the resp.~north poles of the elliptic families.  Since
$$\int_{\alpha_{t/s}^{(2)}}\omega^{(2)}_{t/s}=2\pi\ay \sum_{m\geq 0}\binom{2m}{m}^2 \frac{t^m}{s^m}\;\;\;\;\;\;\text{and}\;\;\;\;\;\;\int_{\alpha_s^{(1)}}R_1=2\pi\ay\left(\pi\ay+\log(s)+\sum_{n>0}\binom{2n}{n}^2 \frac{s^n}{n}\right),$$
we get (with $S^1_{\epsilon}$ counterclockwise from $-\pi$ to $\pi$)
\begin{align*}
\int_{\varepsilon_{0,t}}R &=  (2\pi\ay)^2\oint_{S^1_{\epsilon}}\left\{\left(\pi\ay+\log(s)+\sum_{n>0}\binom{2n}{n}^2\frac{s^n}{n}\right)\left(\sum_{m\geq 0}\binom{2m}{m}^2\frac{t^m}{s^m}\right)\frac{ds}{s}\right\}
\\
&\hspace{1cm}-(2\pi\ay)^3\oint_{S^1_{\epsilon}\times\alpha^{(1)}}\delta_{[0,\frac{1}{16}]\times \beta^{(1)}}\left(\pi\ay+\log\left(\frac{t}{s}\right)+\sum_{m>0}\binom{2m}{m}^2\frac{t^m}{ms^m}\right)
\\
&=(2\pi\ay)^3\sum_{n>0}\binom{2n}{n}^4\frac{t^n}{n}+(2\pi\ay)^2\oint_{S^1_{\epsilon}}\log(s)\sum_{m\geq 0}\binom{2m}{m}^2\frac{t^m}{s^m}\frac{ds}{s}
\\
&\hspace{1cm}+(2\pi\ay)^3\left(\pi\ay+\log\left(\frac{t}{-\epsilon}\right)+\sum_{m>0}\binom{2m}{m}^2\frac{t^m}{m(-\epsilon)^m}\right)+(2\pi\ay)^3\pi\ay
\\
&=(2\pi\ay)^3\left(\pi\ay+\log(t)+\sum_{n>0}\binom{2n}{n}^4\frac{t^n}{n}\right)
\end{align*}
using $\oint \log(s) s^{-m}\frac{ds}{s}=2\pi\ay \log\epsilon$ for $m=0$ and $(-1)^{m-1}\frac{2\pi\ay}{m\epsilon^m}$ for $m>0$.  Taking imaginary parts,  this recovers the answer we already have in \eqref{e6a}.  Note that in this case we really did need both terms in the Deligne-Beilinson cup product and not just $R_1\Omega_2$.

%%%%%%%%%%%%%%%
%%%%%%%%%%%%%%%

\section{The Mellin-Barnes trick and Beilinson for $K_2$}\label{S8}

In this section, we consider the Hadamard product of the 2:1-basechanges of two hypergeometric families of elliptic curves, with $\fa^{(1)}=(\tfrac{1}{4},\tfrac{3}{4})$ and $\fa^{(2)}=(\tfrac{1}{2},\tfrac{1}{2})$ (and $\fb=(1,1)$ for both).  In this way we get a family of hypergeometric CY3 motives with data $\fa=(\tfrac{1}{4},\tfrac{1}{2},\tfrac{1}{2},\tfrac{3}{4})$, $\fb=(1,1,1,1)$ and Hodge numbers $(1,1,1,1)$.  Moreover, by modifying the cup product of a $K_0$-cycle on the first elliptic family and a $K_2$-cycle on the second, we get a $K_2$-class on the CY3 motive whose regulator is given by a modified Deligne-Beilinson product.  This is enough to get a formula for one regulator period about $t=0$, whereupon analytic continuation via Mellin-Barnes can be applied to get formulas for all regulator periods about $t=\infty$.  The main theorem of this section is thus a formula for the regulator determinant on a neighborhood of $\infty$, which we can then use to check the Beilinson Conjecture.

\subsection{Cycles via Hadamard products}\label{Shad}

Consider the families $\mathcal{E}^{(j)}\to \PP^1_s$ given by Kodaira compactifications of $1-s\phi^{(j)}(x_j,y_j)=0$ with $$\phi^{(1)}=(y_1-(x_1-x_1^{-1}))^2/y_1\;\;\;\text{ and }\;\;\; \phi^{(2)}=(x_2-x_2^{-1})(y_2-y_2^{-1}),$$
with singular fibers $\mathrm{I}_4,\mathrm{I}_2,\mathrm{I}_2,\mathrm{I}_4$ resp.~$\mathrm{I}_4,\mathrm{I}_1,\mathrm{I}_1,\mathrm{I}_0^*$. 
The periods of the residue 1-forms $$\mu_s^{(j)}=\frac{1}{2\pi\ay}\omega^{(j)}_s=\frac{1}{2\pi\ay}\mathrm{Res}\left(\frac{dx_j/x_j\wedge dy_j/y_j}{1-s\phi^{(j)}}\right)$$ over the vanishing cycles at $s=0$ are given by $\sum_{n\geq 0}(-1)^n\binom{4n}{2n}\binom{2n}{n}s^{2n}$ resp.~$\sum_{n\geq 0}\binom{2n}{n}^2s^{2n}$.  As in the last section, we take $\tx_{\st}\overset{\pi}{\to} \PP^1_s$ to be a resolution of $X_{\st}=\mathcal{E}^{(1)}\times_{\PP^1}\rho_{\st}^*\mathcal{E}^{(2)}$ (where $\st=\sqrt{t}$); in particular, $\pi^{-1}(0)$ can be taken to be the blowup of $\mathrm{I}_4\times\mathrm{I}_4$ at the 16 ``vertices'' (which produces exceptional quadric surfaces).

Let $Z^{(2)}$ be the extension of the symbol $\{g_1,g_2\}:=\{\tfrac{1+y_2}{1-y_2},\tfrac{x_2-1}{x_2+1}\}$ to an element of the higher Chow group $\mathrm{CH}^2(\mathcal{E}^{(2)}\setminus E_{\infty}^{(2)},2)$.  This cycle is the pullback of $\{x_2,y_2\}$ under the obvious involution of the total space (sending $s\mapsto \tfrac{1}{16s}$).

Next, define $Z^{(1)}\in \mathrm{CH}^1(\mathcal{E}^{(1)},0)$ to be $2[(-1,0)]-2[(1,0)]$, and take (for $|s|<\tfrac{1}{8}$) the 1-chain $\Gamma_s$ to be twice the counterclockwise path on $E^{(1)}_s$ between these points, limiting (as $s\to 0$) to a path on the component $y_1=0$ of the $\mathrm{I}_4$ fiber over $0$.  Then $\partial\Gamma_s=Z^{(1)}_s$, and following the method of \cite[Example 10.12(b)]{Ke}, one computes
\begin{align*}
\int_{\Gamma_s}\mu^{(1)}_s&= 2\int_1^{-1}\frac{1}{2\pi\ay}\oint_{|y_1|=\epsilon}\sum_{k\geq 0} s^k(\phi^{(1)})^k\frac{dy_1}{y_1}\frac{dx_1}{x_1} \\&= 2\int_1^{-1}\sum_{k\geq 0}s^k [(\phi^{(1)})^k]_{(y_1)^0}\frac{dx_1}{x_1}=2\int_1^{-1}(1+2(x_1^{-1}-x_1)s+6(x_1^{-1}-x_1)^2s^2+\cdots)\frac{dx_1}{x_1}\\
&=-2\pi\ay+16s+24\pi\ay s^2+\mathcal{O}(s^3).
\end{align*}
Since we know by \cite[Prop.~5.1]{GKS} that applying $L^{(1)}=D^2+64s^2(D+\tfrac{1}{2})(D+\tfrac{3}{2})$ to this ``truncated normal function'' \emph{must} yield a constant multiple of $s$, we conclude that it yields exactly $16s$.  Thus $Z^{(1)}_s$ is generically nontorsion.  

On the other hand, $Z^{(1)}_0$ is trivial (twice a 2-torsion cycle), not just in the Chow group of $E^{(1)}_0\cong \mathrm{I}_4$, but in its motivic cohomology, as it is the divisor of the rational function $f(x_1):=(x_1+1)^2/(x_1-1)^2$ (on the component $\{y_1=0\}\subset \mathrm{I}_4$), which is $1$ at $x_1=0$ and $\infty$.  Now consider the cup product $\pi_1^*Z^{(1)}\cup \pi_2^*\rho_{\st}^*Z^{(2)}$ in $\mathrm{CH}^3(X_{\st}\setminus(\mathrm{I}_4\times\mathrm{I}_4),2)$.\footnote{Here $\mathrm{I}_4\times\mathrm{I}_4$ is just the fiber of the unresolved $X_{\st}$ over $s=0$.}  We can extend it as a precycle in $Z^3(X_{\st},2)$, with boundary on $\mathrm{I}_4\times\mathrm{I}_4$ of the form $Z_0^{(1)}\times \text{Res}(\rho_{\st}^*Z^{(2)})$ in $\{y_1=0\}\times \mathrm{I}_4$.  Adding the precycle $\{f\}\times \text{Res}(\rho_{\st}^*Z^{(2)})\in \text{im}\{Z^2(\{y_1=0\}\times\mathrm{I}_4,2)\to Z^3(X_{\st},2)\}$ to the cup-product precycle gives something closed on $X_t$, i.e.~in $\mathrm{CH}^3(X_{\st},2)$.  

But this is not enough to yield a higher cycle on the smooth model.  Fortunately, its pullback to $\{(0,0),(\infty,0)\}\times \mathrm{I}_4\subset \mathrm{I}_4\times\mathrm{I}_4$ vanishes identically thanks to $f$ being $1$ there.  Hence this yields an element of motivic cohomology $H^4_{\text{Mot}}(X_{\st},\ZZ(3))$, which pulls back to $H^4_{\text{Mot}}(\tx_{\st},\ZZ(3))\cong \mathrm{CH}^3(\tx_{\st},2)$.  (Basically, one needs the nontorsion Mordell-Weil class of the first family to restrict to torsion in motivic cohomology at the south pole, which is also true\footnote{Indeed, this is essentially forced by the monodromy exponents.  Though there is the issue of how to modify the hypergeometric pullbacks to ensure the cycle is defined over $\QQ$, as we did here.} in the 2 other hypergeometric examples enumerated in \S\ref{S5}.)  Call this cycle $Z_{\st}$.

For the triples of currents attached to $Z_{\st}$, the technical result we need\footnote{to appear as a special case of joint work of the authors with D.~Akman} is that the Deligne-Beilinson cup-product (of triples associated to $Z^{(1)}$ and $Z^{(2)}$) plus the term from the precycle over $s=0$ plus a coboundary in the cone complex equals $$(0,0,R'):=(0,0,R_1'\Omega_2\pm2\pi\ay\log(f(x_1))\text{Res}_{\mathrm{I}_4}(\Omega_2)\delta_{\{y_1=0\}\times\mathrm{I}_4}),$$ where 
\begin{align*}
\Omega_2 &=\text{dlog}(g_1)\wedge\text{dlog}(g_2)=\tfrac{4}{\phi^{(2)}}\tfrac{dx_2}{x_2}\wedge\tfrac{dy_2}{y_2}=4\tfrac{\st}{s}\tfrac{dx_2}{x_2}\wedge\tfrac{dy_2}{y_2}\\
&=4\tfrac{\st}{s}\rho_{\st}^*\omega^{(2)}\wedge (-\tfrac{ds}{s})=-4\st \rho_{\st}^*\omega^{(2)}\wedge \tfrac{ds}{s^2}
\end{align*}
and $R_1'$ is a current representing the lift of the $K_0$ Abel-Jacobi class guaranteed by Theorem \ref{t3a}.  Since we have $L^{(1)}\int_{\Gamma_s}\mu^{(1)}_s=16 s$, and $\mathsf{Q}_0=\tfrac{1}{4}$, the theorem says that $\nabla_D[R_1']_{E^{(1)}_s}=4 s\omega^{(1)}_s$.  Moreover, $R_1'$ restricts to $d[\log(f)]$ at $s=0$ so its class is zero there, and thus $R_1'=4 \int_0 \omega_s^{(1)}ds$.  Finally, the second term in $R'$ can be ignored for our purposes since it is supported over $s=0$ and the integration cycle $\varepsilon_{0,\hat{t}}$ is supported over $|s|=\epsilon$ (where $0<|t|\ll \epsilon \ll 1$ and we use the clockwise orientation).

Which is all to say that the final ``Hadamard regulator product'' takes the simple form:
\begin{align*}
\int_{\varepsilon_{0,\st}}R'&= \int_{\varepsilon_{0,\st}}R_1'\Omega_2=4\oint_{S_{\epsilon}^1}\left(\int_{\alpha_s^{(1)}}R_1'\right)\left(\int_{\alpha_{\st/s}^{(2)}}\st\omega^{(2)}_{\st/s}\right)\frac{ds}{s^2} \\
&= 16(2\pi\ay)^2\oint_{S_{\epsilon}^1}\left(\int_0^s \sum_{n\geq 0}(-1)^n\textstyle\binom{4n}{2n}\binom{2n}{n}\tilde{s}^{2n}d\tilde{s}\right)\left(\sum_{m\geq 0}\textstyle\binom{2m}{m}^2\frac{\st^{2m+1}}{s^{2m}}\right)\frac{ds}{s^2}\\
&=16(2\pi\ay)^2 \oint_{S_{\epsilon}^1}\left(\sum_{n\geq 0}\textstyle(-1)^n\binom{4n}{2n}\binom{2n}{n}\frac{s^{2n+1}}{2n+1}\right)\left(\sum_{m\geq 0}\textstyle\binom{2m}{m}^2\frac{\st^{2m+1}}{s^{2m+1}}\right)\frac{ds}{s}\\
&=16(2\pi\ay)^3\sum_{n\geq 0}(-1)^n\binom{4n}{2n}\binom{2n}{n}^3\frac{\st^{2n+1}}{2n+1}.
\end{align*}
Taking $D_{\st}$ of this gives $16\st$ times a 2:1 pullback of the standard hypergeometric period\footnote{The factor of $(2\pi\ay)^3$ makes this the period of the $\QQ$-de Rham residue form $\omega$.} for $(\fa,\fb)=((\tfrac{1}{4},\frac{1}{2},\tfrac{1}{2},\tfrac{3}{4}),(1,1,1,1))$, which is exactly as expected.

However, there are two important things to notice here.  First, the pullback is by $\th\mapsto -\th^2$ (not $\th^2$), i.e.~pullback composed with an ``imaginary rotation''.  Next, the $\tx_{\st}$ on which we have constructed our $K_2$-cycle is not birational to a CY.  We must take its quotient $Y_{\st}$ by the order-2 automorphism $$(x_1,y_1,x_2,y_2,s)\mapsto (-x_1,-y_1,x_2,-y_2,-s).$$  This is at least birational to the Hadamard product of the original hypergeometric families, with $\mathrm{rk}((\mathrm{Gr}^W_3H^3(Y_{\st}))^{3,0})=1$.  Since (twice) the cycle $\varepsilon_{0,\st}$ is pulled back from $Y_{\st}$, the push-forward of $Z_{\st}$ to $\mathrm{CH}^3(Y_{\st},2)$ has nontorsion regulator and gives the correct extension.

Finally, what we are really interested in is the hypergeometric motive $M_z=M_{2^{10}t}$ with period $\sum_{n\geq 0}\binom{4n}{2n}\binom{2n}{n}t^n$.  Writing $\mathcal{M}_z$ for the corresponding Hodge structure, we have inclusions $\mathcal{M}_{-2^{10}\th^2}\subset \mathrm{Gr}^W_3H^3(Y_{\st})\subset H^3(\tx_{\st})$.  For certain nonzero rational values $\hat{t}=q\in \QQ^{\times}$ (which include the positive integers and are identified below), the cycle $Z_q\in H_{\text{Mot}}^4(\tx_q,\QQ(3))$ we have constructed extends to an integral model, hence should satisfy the Beilinson Conjecture for $M_{-2^{10}q^2}$ (not $M_{2^{10}q^2}$, due to the ``rotation'').  The regulator period computed above is unsuitable in two senses:  the series does not converge where we want to compute the regulator (for large $q$), and the fact that it is imaginary where it is defined (for small $q$) means that $\varepsilon_0$ is anti-invariant under $\mathsf{F}_{\infty}$.  Here passing to the real regulator actually means taking the real part of the regulator periods, which would be zero for this one!

To state the Beilinson Conjecture for this case, let $\omega_{\hat{t}}$ be the $\QQ$-dR form with period $(2\pi\ay)^3\sum_{n\geq 0}(-1)^n \binom{4n}{2n}\binom{2n}{n}\hat{t}^{2n}$ and $r_{{\hat{t}}}=\mathrm{Re}(R'_{\hat{t}})$ the real regulator class.  Let $\delta_1,\delta_2$ be a basis of the $\mathsf{F}_{\infty}$-invariant classes in $H_3(Y_{\hat{t}},\QQ)$.  Then we should have
\begin{equation}\label{eDET}
L'(M_{-2^{10}q^2},1)\sim \frac{1}{(2\pi\ay)^4}\left| \begin{matrix}\int_{\delta_1}\omega_{q} & \int_{\delta_1} r_{q} \\ \int_{\delta_2}\omega_{q} & \int_{\delta_2} r_{q} \end{matrix}\right|,
\end{equation}
in which all of the matrix entries are real.

\subsection{Computation of regulator periods}

In this final subsection, we work out the regulator determinant \eqref{eDET} in Beilinson's Conjecture for the $K_2$-element in the $(\fa,\fb)=((\tfrac{1}{4},\frac{1}{2},\tfrac{1}{2},\tfrac{3}{4}),(1,1,1,1))$-hypergeometric family.  Above, we had to construct the family of cycles on the 2:1-basechange of the pullback of the family by $z\mapsto -z$ (where $z$ is the hypergeometric parameter, and $t=2^{-10}z$ the algebro-geometric parameter).  The integral regulator class $R$ satisfies $DR=8\sqrt{t}\omega=\tfrac{\sqrt{z}}{4}\omega$, where $\omega$ was the standard $\QQ$-dR residue form, and a computation showed that it had a period of the form in the first line of
\begin{align*}
R_0&=16(2\pi\ay)^3\sum_{n\geq 0}(-1)^n\binom{4n}{2n}\binom{2n}{n}^3\frac{\sqrt{t}^{2n+1}}{2n+1}\\
&=\tfrac{1}{4}(2\pi\ay)^3\sum_{n\geq 0}(-1)^n\frac{[\frac{1}{4}]_n[\frac{1}{2}]_n^2[\frac{3}{4}]_n}{n!^4(n+\frac{1}{2})}z^{n+\frac{1}{2}}\\
&=\tfrac{1}{4}(2\pi\ay)^3\int_{\mathscr{C}_{\textup{right}}}\frac{\Gamma(-s)\Gamma(1+4s)\Gamma(1+2s)z^{s+\frac{1}{2}}}{\Gamma(1+s)^52^{10s}(s+\frac{1}{2})}\frac{ds}{2\pi\ay}.
\end{align*}
In order to change the Mellin-Barnes style integral from a right-half-plane loop to a left-half-plane loop, it is convenient to introduce a few expressions. Begin by setting
$$\bbg^{\alpha}_k(\fa,\fb):=\prod_i \frac{\Gamma(\fb_i-1+\alpha+k)}{\Gamma(\fa_i+\alpha+k)},$$
and write for short $\bbg^{\alpha}_k:=\bbg^{\alpha}_k((\tfrac{1}{4},\tfrac{1}{2},\tfrac{1}{2},\tfrac{3}{4}),(1,1,1,1))$.  Put
\begin{align*}
\ta&:= \sum_{k\geq 0}\frac{(-1)^k\bbg^{\frac{1}{4}}_{k}}{k-\frac{1}{4}}z^{-k}, \;\;
\tb:= \sum_{k\geq 0}\frac{(-1)^k\bbg^{\frac{3}{4}}_{k}}{k+\frac{1}{4}}z^{-k},\;\;
\tc:= \sum_{k > 0}\frac{(-1)^k\bbg^{\frac{1}{2}}_{k}}{k}z^{-k}\\
\td&:= \sum_{k\geq 0}\frac{(-1)^k\bbg^{\frac{1}{2}}_{k}}{k}\left\{4H_{4k+1}-10H_{2k}+6H_k+\frac{1}{k}\right\}z^{-k}
\end{align*}
and
$$\bba:= \sum_{k\geq 0}{(-1)^k\bbg^{\frac{1}{4}}_{k}}z^{-k}, \;\;
\bbb:= \sum_{k\geq 0}{(-1)^k\bbg^{\frac{3}{4}}_{k}}z^{-k},\;\;
\bbc:= \sum_{k \geq 0}{(-1)^k\bbg^{\frac{1}{2}}_{k}}z^{-k};$$
and note that
$$D\ta z^{\frac{1}{4}}=-\bba z^{\frac{1}{4}},\;\;D\tb z^{-\frac{1}{4}}=-\bbb z^{-\frac{1}{4}}, \;\; D(\tc-\bbg^{\frac{1}{2}}_0\log(z))=-\bbc$$ 
where $\bbg^{\frac{1}{2}}_0=2\sqrt{2}\pi$.

Computing the left-half-plane integral now gives
$$
R_0 = \ay\pi\ta z^{\frac{1}{4}}-\ay\pi\tb z^{-\frac{1}{4}}-2\sqrt{2}\ay \log(4z)\tc -2\sqrt{2}\ay\td+4\pi\ay(\log(4z)+4)^2,
$$
where we are allowed to work modulo terms in $\QQ(3)=(2\pi\ay)^3\QQ$. (After all, the regulator class is only defined modulo $\MM(3)$, where $\MM$ is the local system, and we are ultimately interested in real parts of periods.) Analytically continuing $R_0$ around the origin twice and adding (call the result S) eliminates the A and B expressions, whereupon analytically continuing twice and subtracting yields\footnote{This $R_1$ can also be written as a Frobenius deformation.  It is the coefficient of $s^0$ in $-16\sqrt{2}\pi\sum_{k\geq 0}\frac{(-1)^{k+s}}{k+s}\bbg^{\frac{1}{2}}_{k+s}z^{-(k+s)}$.}
$$R_1 := -16\sqrt{2}\pi\tc+64\pi^2(\log(4z)+4).$$
Subtracting this from twice S yields
$$R_2:=8\sqrt{2}\ay \log(4z)\tc+8\sqrt{2}\ay\td+256\pi\ay-16\pi\ay(\log(4z)+4)^2.$$
If we subtract rather than added at the beginning of this process, then double, then subtract $R_1$, we get
$R_3:=4\pi\ay\ta z^{\frac{1}{4}}-4\pi\ay\tb z^{-\frac{1}{4}}.$
Finally, analytically continuing $R_3$ \emph{once} counterclockwise in $z$ gives a regulator period of \emph{minus} the cycle (since it is multivalued), which is a period of the $K_2$-class over minus the analytically continued 3-cycle:
$$R_4:=4\pi\ta z^{\frac{1}{4}}+4\pi\tb z^{-\frac{1}{4}}.$$
The entire process is guaranteed to give periods over integral 3-cycles.  Notice that $R_1$ and $R_4$ are real and so \emph{are} the real regulator periods $r_1$ and $r_4$ for real $z$ with $|z|>1$.

To get the periods of $\omega$ over the same integral 3-cycles, we use $\frac{4}{\sqrt{z}}DR_j=\omega_j$.  Writing $\hat{R}_j:=(2\pi\ay)^{-2}R_j$ and $\hat{\omega}_j:=(2\pi\ay)^{-2}\omega_j$, the desired matrix (with determinant the right-hand side of \eqref{eDET}) is then
\begin{equation}\label{eDETK2}
\mathfrak{R}_z:=\left(\begin{matrix}\hat{R}_1 & \hat{\omega}_1 \\ \hat{R}_4 & \hat{\omega}_4\end{matrix}\right)=\left(\begin{matrix} \frac{4\sqrt{2}}{\pi}\tc-16(\log(4z)+4) & -\frac{16\sqrt{2}}{\pi}\bbc z^{-\frac{1}{2}} \\ -\frac{1}{\pi}(\ta z^{\frac{1}{4}}+\tb z^{-\frac{1}{4}}) & \frac{4}{\pi}(\bba z^{-\frac{1}{4}}+\bbb z^{-\frac{3}{4}}) \end{matrix}\right).
\end{equation}
By virtue of Remark \ref{rZ}, one expects that the cycle ``upstairs'' extends to an integral model for $\sqrt{t}=n\in \ZZ\setminus\{0\}$, which gives $z=2^{10}n^2$.  Keeping in mind that, due to the $(-1)^n$ in the series, our Calabi-Yau 3-fold $X_z$ (on which we have our cycle) has $H^3(X_z)=M_{-z}$ where $M$ is the hypergeometric family of motives, \eqref{eDET} predicts that $L'(M_{-2^{10}n^2},1)/\det(\mathfrak{R}_{2^{10}n^2})\in \QQ^{\times}$ for (apparently) any $n$.

In fact, we can both improve and simplify this prediction.  

\begin{prop}\label{p10a}
The cycle $Z_{\st}$ extends to an integral model $\fx_{\st}$ of $\tx_{\th}$ when $\st=n/2^o$ \textup{(}with $n\in \ZZ_{>0}$ and $o\in \ZZ_{\geq 0}$\textup{)}, i.e.~$z=2^{10-2o}n^2$ or $t=n^2/4^o$.
\end{prop}
\begin{proof}
Recall that $Z_{\th_0}$ is constructed on a 3-fold fibered over $\PP^1_s$ with fibers $E^{(1)}_s\times E^{(2)}_{\th/s}$.  It is given by $\pi_1^*Z^{(1)}\cup\pi_2^*\rho_{\th}^*Z^{(2)}$ plus a component over $s=0$, whose boundaries (which cancel) are $\pm 2$ times $([(-1,0)]-[(1,0)])\times \mathrm{Res}(\rho_{\th_0}^*{(2)})$.  At first glance, the potential problem with reduction mod $p$ is that $\th$ reduces to $\infty$ so that the residue of the first component (which comes from $E^{(2)}_{\infty}$) occurs not just over $s=0$ but over all of $\PP^1_s$.  So this would restrict us to $\th\in \ZZ$.  However, since the boundary (of each component) is \emph{twice} an integral cycle, it vanishes identically if we reduce mod $2$.  So we do not have to worry about $\th$ reducing to $\infty$ mod $2$, as claimed.
\end{proof}

For the simplified formula for the regulator determinant, we can repackage \eqref{eDETK2} in the following

\begin{thm}\label{tK2}
For $z\in \RR_{>1}$, write $$\Gamma^{\alpha}_k:=\prod_{i=1}^4\frac{\Gamma(\alpha+k)}{\Gamma(\alpha+k+\fa_i)},\;\; S_{\alpha}(z):=\sum_{k\geq 0}\Gamma^{\alpha}_kz^{-k},\text{ and }R_{\alpha}(z):=\sum_{k\geq 0}\Gamma^{\alpha}_k\frac{z^{-k}}{k-\frac{1}{2}+\alpha}$$ 
\textup{(}with $\sum_{k>0}$ in $R_{\alpha}$ if $\alpha=\tfrac{1}{2}$\textup{)}.  Then the real regulator determinant associated to $M_{-z}$, normalized as in RHS\eqref{eK2}, is $$r(t):=\det\left(\begin{matrix}  4(\log(4z)+4) -\frac{\sqrt{2}}{\pi}R_{\frac{1}{2}}(z) & \frac{4\sqrt{2}}{\pi\sqrt{z}}S_{\frac{1}{2}}(z) \\ -\frac{1}{4\pi}(z^{\frac{1}{4}}R_{\frac{1}{4}}(z)+z^{-\frac{1}{4}}R_{\frac{3}{4}}(z)) & \frac{1}{\pi}(z^{-\frac{1}{4}}S_{\frac{1}{4}}(z)+z^{-\frac{3}{4}}S_{\frac{3}{4}}(z)) \end{matrix}\right),$$
where $z=2^{10}t$.
\end{thm}
After lengthy processing,\footnote{some of it on the MIT cluster; we thank Edgar Costa for helping us to implement this!} MAGMA and Pari tell us that
$$\frac{L'(1)}{r(t)} \;\overset{\text{\color{red}num}}{=}\;\frac{1}{64},\text{-}\frac{1}{16},\frac{1}{640},3,\text{-}40,\text{-}248\;\;\;\text{ for }\;\;\;t=   \frac{1}{16},\frac{1}{4},1,9,25,49,$$
where $L(s)$ means $L(M_{-2^{10}t},s)$.
We have also run this computation for numerous rational values of $z$ that do not satisfy Proposition \ref{p10a}.  None of them produced a rational number.

\vspace{2mm}

$$*\;\;\;*\;\;\;*$$

\vspace{2mm}

\begin{thx}
We thank Johannes Walcher and the members of the International Groupe de Travail  on differential equations in Paris for stimulating discussions, John Voight for Remark \ref{rVoight}, Edgar Costa for assistance in implementing a computation on the MIT computer cluster, and Salim Tayou for helpful comments.  VG thanks the Abdus Salam International Centre for Theoretical Physics
for its hospitality in 2023-2024, and the Simons Foundation for making his
stay in Trieste possible. VG acknowledges the school and workshop "Number Theory and Physics" at ICTP
in June-July 2024 where the final stages of this work were carried out for becoming a starting point for a number
of new collaborations. MK's work was partially supported by NSF Grant DMS-2101482, a Simons Foundation travel grant, and a visit to MPIM-Bonn in 2023.
\end{thx}
\pagebreak
%%%%%%%%%%%%%%%
%%%%%%%%%%%%%%%
%%%%%%%%%%%%%%%
%%%%%%%%%%%%%%%

\appendix

\section{$K_2$--regulators ``without cycles''}\label{appB}

Here we briefly mention another numerical verification of the Beilinson conjecture for $\CH^3(M_t,2)$ of fibers of a hypergeometric family of CY3 motives.  The point is the existence of a simple ansatz which apparently allows us to bypass varieties and cycles and get straight to regulator volumes.  While we offer no proofs here, the method is closely related to the more detailed calculations in \S\ref{S8}.  

To make our claim below a bit easier to swallow, we return once more to the philosophy of the Introduction and consider first some families of elliptic curves.  Denote
$$\mathbf{\Gamma}(s) = -\frac{1}{2} \cdot \frac{\Gamma(6s+1)\Gamma(s+1)^3}{\Gamma(2s+1)^3 \Gamma(3s+1)(3s-1/2)} 27^{-s},  \quad \widetilde{\mathbf{\Gamma}}(s) =  (2 \pi \ay)  \frac{\mathbf{\Gamma}(s+1/2) }{\mathbf{\Gamma}(1/2) }$$
and note that $\mathbf{\Gamma}(0) = 1, \, \mathbf{\Gamma}(1/2) = -\frac{1}{2} \cdot \frac{\pi}{\sqrt{27}}.$ Put
$$\Pi_{0}(t) = \sum_{n \in 1/2 + \ZZ_{\ge 0}} {\mathbf{\Gamma}}(n)  t^n
=-\frac{1}{2} \cdot \frac{\pi}{\sqrt{27}}\, t^{1/2} \cdot (
1
+ \frac{2}{9}\*t
+ \frac{10}{81}\*t^2
+ \frac{560}{6561}\*t^3
+ \frac{3850}{59049}\*t^4+ \dots)$$
and
$$\Pi_{1}(t) =  \sum_{n \in \ZZ_{\ge 0}} {\widetilde{\mathbf{\Gamma}}(n)  }t^n =2 \pi \ay \sum_{n \in  \ZZ_{\ge 0}} \frac{\mathbf{\Gamma}(n+1/2) }{\mathbf{\Gamma}(1/2) }t^n
=2 \pi \ay \cdot (
1
+ \frac{2}{9}\*t
+ \frac{10}{81}\*t^2
+ \frac{560}{6561}\*t^3
+ \frac{3850}{59049}\*t^4+ \dots).$$

The hypergeometric functions $\Pi_0 (t) , \Pi_1 (t)$
are periods in the families of elliptic curves\footnote{
	E(t)=ellinit([0,-27/4/t,0,27/2/t,-27/4/t])
	E1(t)=ellinit([1,0,1/27*t,0,0])
}
$$E_0(t) : \; y^2 = x^3 - \frac{27}{4} t^{-1} x^2 + \frac{27}{2} t^{-1} x - \frac{27}{4} t^{-1}\;\;\;\;\text{and}\;\;\;\;\;E_1(t) : \; y^2 + xy + \frac{1}{27} t  = x^3 $$
which are quadratic twists of each other. This makes the periods differ by a simple factor, but it is precisely this difference that is our central point: informally speaking, the ambiguity that arises in computing  the integrals  $\int \Pi_{0} \frac{dt}{t}$ resp.  $\int \Pi_{1} \frac{dt}{t}$ is measured by $(2 \pi \ay)  \mathbf{\Gamma}(0) = 2 \pi \ay $ resp.~$(2 \pi \ay) {\widetilde{\mathbf{\Gamma}}(0) } = (2 \pi \ay)^2 $ by a version of Cauchy's theorem. The distinction between the weight gaps of 1 and 3 in the respective mixed Hodge structures translates into the distinction between the types of algebraic classes whose regulators they encode: $\CH^1 (E_0(t),0)$ in the former case vs.~$\CH^2 (E_1(t),2)$ in the latter! We leave to the reader as an exercise the task of finding actual cycles whose regulators are given by the primitives  $\int^x \Pi_{0} \frac{dt}{t}, \, \int^x \Pi_{1} \frac{dt}{t}$.

A pair of hypergeometric indices $(\fa,\fb)$ is of ``principal type'' if there is no interlacing:  listing the indices in order, either the $\fa_i$ or the $\fb_i$ occur in an unbroken string.  In this case, all the Hodge numbers in the corresponding hypergeometric variation $\mathcal{H}=\mathcal{H}_{\fa,\fb}$ are $1$.  Let $\kappa_{\fa,\fb}$ denote the number of integer indices (i.e., with our conventions, the number of times ``$1$'' occurs).  In the elliptic examples just described, we have $\kappa_{(\frac{1}{6},\frac{5}{6}),(\frac{1}{2},\frac{1}{2})}=0$ and $\kappa_{(\frac{1}{3},\frac{2}{3}),(1,1)}=1$.  In general, without passing to double covers, we have a natural extension of $\mathcal{H}$ by the Tate class $\IH^1(\Gm,\mathcal{H})$.  Arguing as in the previous paragraph using the more uniform definition of $\mathbf{\Gamma}(s)$ in \eqref{eV*} above and the duplication formula $\Gamma (1-z)\Gamma (z)={\frac {\pi }{\sin \pi z}}$,\footnote{Alternatively, one can use the theory of admissible normal functions and formula \eqref{e3f} to see this.} one easily shows that \emph{the weight gap in a principal hypergeometric family of motives defined over $\QQ$ is $1+\kappa_{\fa,\fb}$}.

Turning to the promised CY 3-fold example, we take $(\fa,\fb)=((\tfrac{1}{5},\tfrac{2}{5},\tfrac{3}{5},\tfrac{4}{5}),(\tfrac{1}{6},\tfrac{5}{6},1,1))$, form the $S_n(t)$ as in \eqref{eV**}, and set
	$$r(t) :=  \mathop{\mathrm{det}} \mathop{\mathrm{Re}} \begin{pmatrix}
	S_0(t) & S_0'(t) \\ 
	
	S_1(t) & S'_1(t) 	
\end{pmatrix}.$$
Clearly we have a weight-gap of $3$ between the hypergeometric VHS $\mathcal{H}$ and $\IH^1(\Gm,\H)$, which predicts a $K_2$-type normal function.  In analogy to \S\ref{S8}, we make the following conjectural

\medskip

\sc Claim. \rm \emph{There exists a family of CY 3-folds $X_t$ with middle cohomology containing $\mathcal{H}$. Moreover:}

 (i) \emph{For every $t \in \QQ \cap (0,3125/432)$, there exists a higher cycle $Z_t \in \CH^3(X_t,2)$ whose real regulator volume \textup{(}in the sense of RHS\eqref{eK2}\textup{)} is given by $r(t)$.}
 
%admits the following hypergeometric expression:
%$$r(C_t) = \frac{R(t)}{S'_0 (t)}$$

(ii) \emph{If in addition $t \in \ZZ$, then the cycle $Z_t$ extends to a cycle in $\CH^3(\mathfrak{X}_t,2)$ of an integral model.}

\medskip

\noindent Clearly we cannot construct the family or the cycle by Hadamard products, so one expects that $Z_t$ will be harder to realize than the cycle of \S\ref{S8}.  On the other hand, for the $L$--ratios one has numerically 
 $$\;\;\;\;\;\;\frac{L'(M_t,1)}{r(t)}\overset{\text{\color{red}num}}{=} -\frac{2}{5}, -4, 12,-2,24,40,-68/7\;\;\;\text{at}\;\;\;t=1,2,3,4,5,6,7,$$
which simultaneously gives evidence for the Claim and for Beilinson's Conjecture.

\pagebreak
\curraddr{${}$\\
\noun{ICTP Math Section}\\
\noun{Strada Costiera 11, Trieste 34151 Italy}
\email{${}$\\
\emph{e-mail}: vasily.v.golyshev@gmail.com}

\curraddr{${}$\\
\noun{Department of Mathematics, Campus Box 1146}\\
\noun{Washington University in St. Louis}\\
\noun{St. Louis, MO} \noun{63130, USA}}
\email{${}$\\
\emph{e-mail}: matkerr@wustl.edu}

\end{document}